\documentclass[12pt,a4paper]{amsart}
\usepackage{epsf,epsfig, graphicx, colordvi,amsmath,amsfonts,amssymb}
\usepackage{color}

\usepackage[latin1]{inputenc}

\newtheorem{thm}{Theorem}[section]
\newtheorem{defi}{Definition}[section]
\newtheorem{lem}{Lemma}[section]

\newtheorem{rem}{Remark}[section]

\newcommand{\commentbis}[1]{}
\newcommand{\be}{\begin{eqnarray}}
\newcommand{\ee}{\end{eqnarray}}
\newcommand{\beno}{\begin{eqnarray*}}
\newcommand{\eeno}{\end{eqnarray*}}
\newcommand{\barr}[1]{\begin{array}{#1}}
\newcommand{\earr}{\end{array}}

\newcommand{\conv}{\rightarrow}
\newcommand{\ma}{\alpha}

\newcommand{\mt}{\theta}
\newcommand{\mO}{\Omega}

\newcommand{\cM}{\mathcal M}

\newcommand{\cB}{\mathcal B}

\newcommand{\R}{\mathbb R}

\newcommand{\dt}{{\Delta t}}

\newcommand{\ud}{\frac{1}{2}}  % j+1/2, j+3/2, j-1/2:
\newcommand{\ipud}{{i+\ud}}
\newcommand{\imud}{{i-\ud}}

\newcommand{\pro}{\mathbb{P}_h}

\newcommand{\jpud}{{j+\ud}}
\newcommand{\jmud}{{j-\ud}}

\newcommand{\xia}{x^i_\ma}
\newcommand{\wia}{w^i_\ma}
\newcommand{\xja}{x^j_\ma}
\newcommand{\wja}{w^j_\ma}

\newcommand{\qkell}{{q_{k,\ell}}}

\usepackage{comment}

\title[Convergence of DG schemes for obstacle problems]{Convergence of discontinuous Galerkin schemes for front propagation with obstacles}


\author{Olivier Bokanowski}
\address{Laboratoire Jacques-Louis Lions, Universit\'e Pierre et Marie Curie
75252 Paris Cedex 05 France, and
Universit\'e Paris-Diderot (Paris 7),
5 Rue Thomas Mann 75205, Paris CEDEX 13, France.
}
\email{boka@math.jussieu.fr}
\thanks{Research of the first author is supported by the EU under 
the 7th Framework Programme Marie Curie Initial Training Network
``FP7-PEOPLE-2010-ITN'', SADCO project, GA number 264735-SADCO.}

\author{Yingda Cheng}
\address{Department of Mathematics, Michigan State University,
East Lansing, MI, 48824 USA}
\email{ycheng@math.msu.edu}
\thanks{Research of the second author is
supported by NSF grant DMS-1217563 and the start-up grant from 
Michigan State University.}

\author{Chi-Wang Shu}
\address{Division of Applied Mathematics, Brown University,
Providence, RI 02912, USA}
\email{shu@dam.brown.edu}
\thanks{Research of the third author is
supported by ARO grant W911NF-11-1-0091 and NSF grants
DMS-1112700 and DMS-1418750.}

\date{February 21, 2015}

\begin{document}

\maketitle

\begin{abstract}
We study semi-Lagrangian discontinuous Galerkin (SLDG) and 
Runge-Kutta discontinuous Galerkin (RKDG) schemes for some 
front propagation problems in the presence of an obstacle 
term, modeled by a nonlinear
Hamilton-Jacobi equation of the form $\min(u_t + c u_x, u - g(x))=0$, 
in one space dimension.
New convergence results and error bounds are obtained for 
Lipschitz regular data. 
These ``low regularity" assumptions are the natural ones
for the solutions of the studied equations. 
Numerical tests are given to illustrate the behavior of our schemes.
\end{abstract}

\section{Introduction}
%----------------------------------------

In this paper, we establish convergence of a class of 
discontinuous Galerkin (DG) methods for the one-dimensional 
Hamilton-Jacobi (HJ) equation below, hereafter also called the 
``obstacle" equation,
\begin{subequations}
\label{eq:02}
\be 
  & &  \min( u_t + c \,u_x,\ u-g(x)) = 0, \quad x\in I= (0,1),
\quad t>0, \label{eq:02a} \\
  & & u(0,x) = u_0(x), \quad x\in (0,1), \label{eq:02b}
\ee
\end{subequations}
with periodic boundary conditions on $I$ and a constant 
$c\in \R$.  In \eqref{eq:02}, the function $g$ is called the 
``obstacle" function.

It is well known that taking constraints in optimal control 
problems is not an obvious task. Within the viscosity theory, 
it is possible to devise schemes for obstacle equations such 
as \eqref{eq:02}. However a monotonicity condition is needed 
in general for proving the convergence of the scheme. The 
monotonicity condition can yield a convergence proof of
one-half order in the mesh size \cite{Crandall_Lions_84}
(see also \cite{Bokanowski_Forcadel_Zidani_2010} for more 
specific partial differential equations (PDEs) with an 
obstacle term and related error estimates). 
However, a serious restriction of such monotonicity condition
is that the schemes become at most first order accurate for
smooth solutions or in smooth regions, thus making the schemes
highly inefficient for practical computation.
On the other hand, it is very difficult to show convergence
of formally higher order schemes, such as the
ones studied in this paper, when the solution is not regular enough.

In our previous work \cite{bok-che-shu-14}, we proposed a class of 
Runge-Kutta DG (RKDG) methods adapted to front propagation
problems with obstacles. 
The DG methods under consideration were originally devised
to solve conservation laws, see for example the review paper
\cite{Cockburn_2001_RK_DG}.
As for the DG-HJ solvers,
in \cite{Hu_1999_SIAM_DG_FEM, Li_2005_AML_DG_HamiJaco}, the first 
efforts relied on solving the conservation law system satisfied by 
the derivative of the solution.  
See also \cite{chen2007adaptive} for an adaptive version of this scheme.
In \cite{Cheng_07_JCP_HJ}, a DG 
method for directly solving the Hamilton-Jacobi equation was 
developed and was later generalized to solve front propagation 
problems \cite{Bokanowski_Cheng_Shu_11}
and obstacle problems~\cite{bok-che-shu-14}.
Other direct DG solvers include the central DG scheme 
\cite{Li_Yakovlev_10} and the local DG scheme \cite{Yan_Osher_11}.
The schemes proposed in \cite{bok-che-shu-14} feature a simple treatment 
of the obstacle functions. Stability analysis is performed with
forward Euler, a Heun scheme and a TVD Runge-Kutta third order 
(TVD-RK3) time discretization using the techniques developed in 
Zhang and Shu~\cite{Zhang_Shu_RK3_2010}.

On the other hand, the semi-Lagrangian DG (SLDG) methods were 
proposed in 
\cite{restelli2006semi, rossmanith2011positivity, qiu2011positivity} 
to compute incompressible flow and Vlasov equations, as well as in 
\cite{Bokanowski_Simarmata} for some general linear first and 
second order PDEs.
The advantage of SLDG is its ability to take large time steps 
without a CFL restriction. 
However, it is difficult to design SLDG methods for nonlinear problems
(some SL schemes for Hamilton-Jacobi-Bellman equations were proposed 
in \cite{bokanowski2012adaptive}, but without convergence proof).
For general SL methods, we refer to 
the works of Falcone and Ferretti~\cite{falcone1994discrete, 
falcone1998convergence, falcone2002semi},
see also the textbook~\cite{falconebook}.
There is also a vast literature on high order finite difference schemes
for solving HJ equations, see, e.g. 
\cite{Osher_1991_SIAM_NonOscill, Abgrall_96_CPAM_HJ, Jiang_1999_SIAM_WENO,
Kurganov-Tadmor-HJ}. 

Beyond the scope of the present paper, yet of great interest, 
we also mention the second order PDE with obstacle terms, such 
as $\min(u_t - c u_{xx}, u-g(x))=0$, with $c>0$.
This is the case of the so-called ``American options" in 
mathematical finance \cite{Achdou_Pironneau}. 
Explicit schemes were proposed and proved to converge within 
the viscosity theory (see for instance \cite{Jakobsen_2003}) yet 
with a reduced rate of convergence. Variational methods for nonlinear
obstacle equations can also be devised \cite{Hintermuller_et_al} but 
will lead in general to nonlinear implicit schemes
which are computationally more demanding.

The scope of the present paper is to study convergence of the SLDG 
and RKDG schemes for the obstacle problem \eqref{eq:02}. The main 
challenges include the low regularity of the solution and the 
nonlinear treatment needed to obtain the obstacle solution.  Due 
to the fully discrete nature of the method, traditional 
techniques for obtaining semi-discrete error estimates of DG methods 
for hyperbolic problems cannot directly apply here.
Therefore, fully discrete analysis is necessary. Fully discrete 
analysis of RKDG methods for conservation laws has been performed 
in the literature.
In \cite{Zhang_Shu_RK3_2010}, error estimates for RKDG methods 
for scalar conservation laws were provided for smooth solutions.
In \cite{cockburn2008error, zhang2013error}, discontinuous 
solutions with RK2 and linear polynomials and RK3 time 
discretizations with general polynomials were studied for
linear conservation laws.
However, to our best knowledge, convergence results for second
and higher order DG schemes solving nonlinear hyperbolic equations
with irregular solutions are not available.

As a consequence of our results,
assuming that $h$ is a space step and $\dt$ a time step,
we shall show  error bounds of the order of $O(h^{9/10})$ for 
the SLDG schemes (under time stepping $\dt \equiv C h^{3/5}$,
larger time step can be taken with a lower convergence rate), 
and of order $O(h^{1/2})$ for the RKDG schemes (under time 
stepping $\dt \equiv C h$). We will need a natural ``no shattering" 
regularity assumption 
on the exact solution that will be made precise in the sequel, 
otherwise we will typically 
assume the exact solution to be Lipschitz regular and 
piecewise $C^q$ regular for some $q\geq 1$ for SLDG (resp. 
$q\geq 2$ for RKDG).

The main idea of our proof is based on the dynamic programming 
principles as illustrated below.
The viscosity solution of \eqref{eq:02} also corresponds to the 
following optimal control problem:
\begin{eqnarray}
  \label{eq:0v}
  u(t,x)= \max\bigg(u_0(x- ct),\ \max_{\mt\in[0,t]} g(x-c\mt)\bigg).
\end{eqnarray}
The function $u$ is also the solution of the Bellman's dynamic 
programming principle (DPP): for any $\dt>0$,
\begin{eqnarray}\label{eq:0dpp}
  u(t+\dt,x) = \max \bigg( u(t,x-c\dt) , 
\max_{\mt\in[0,\dt]} g(x-c\mt)\bigg), \quad \forall \,t\geq 0.
\end{eqnarray}
Notice conversely that the DPP \eqref{eq:0dpp}, together with 
$u(0,x)=u_0(x)$, implies \eqref{eq:0v}.

If we denote
$$ u^n(x) := u(t_n,x) $$
and
$$ g_t(x):=\max_{\mt\in[0,t]} g(x-c\mt),$$
then the DPP implies in particular for any $x$ and $n\geq 0$:
\be\label{eq:vn1exact}
  u^{n+1}(x)= \max ( u^n(x-c\dt) , g_{\dt}(x)).
\ee

Using formula \eqref{eq:0v}, we can see that when $u_0$ and $g$ 
are Lipschitz regular, then $u$ is also Lipschitz regular in space 
and time, and in general
no more regularity can be assumed (the maximum of two regular 
functions is in general no more than Lipschitz regular).
When $u_0$ is a discontinuous function (otherwise piecewise 
regular), formula \eqref{eq:0v} implies also some regularity
on the solution $u$.

The rest of the paper is organized as follows.  In Section 
\ref{sec:dg}, we describe the DG methods for the obstacle 
equations.  In Section \ref{sec:prelim}, we collect some 
lemmas that will be used in our convergence proofs.  Section 
\ref{sec:sldg} and Section \ref{sec:rkdg} are devoted to the 
convergence analysis of SLDG and RKDG methods, respectively. In both 
sections, we will first establish error estimates for SLDG 
and RKDG schemes for linear transport equations without obstacles. 
Then we will use DPP illustrated above to prove convergence of 
the numerical solution in the presence of obstacles.
Numerical examples are given in Section~\ref{sec:num}.
We conclude with a few remarks for the multi-dimensional case in 
Section~\ref{sec:conclusion}.

%\bigskip

\section{DG schemes for the obstacle equation}
\label{sec:dg}
%--------------------------------------------

In this section, we will introduce the SLDG and RKDG methods for 
the obstacle equation \eqref{eq:02}. For simplicity of discussion, in the 
rest of the paper, we will assume $c$ to be a positive constant.

Let $I_j:=(x_\jmud,x_\jpud), \,j=1, \ldots, N$ be a set of 
intervals forming a partition of $I=(0,1)$.
We denote $h:=\max_j h_j$ where $h_j=|I_j|$ is the length of 
the interval $I_j$.  For a given integer $k\geq 0$,
let $V_h$ be the DG space of piecewise polynomial of degree at 
most $k$ on each interval $I_j$:
\be
  V_h:= \{ v :I\conv\R,\ \ v_{|I_j} \in P^k,\ \forall j\}.
\ee

To introduce the methods for the obstacle problems, we follow 
two steps. Firstly, we describe the DG solvers for the linear 
advection equation $v_t + c v_x=0$. 

To advance the numerical solution in one step from 
$v_h^n \in V_h$ to $v_h^{n+1} \in V_h$, we consider one of the 
two DG methods described below.

\bigskip

{\underline{SLDG Scheme}:}
\be 
\label{eq:SLDG}
  v_h^{n+1} : = \Pi_h \big( v_h^n (\cdot - c\dt) \big),
\ee
where $\Pi_h$ is the $L^2$ projection onto the space $V_h$.  
We shall denote this SLDG solver by $v_h^{n+1}=G_{\dt}^{SL} (v_h^n)$.

\bigskip

{\underline{RKDG Scheme} (by TVD-RK3 time stepping):}
find $v_h^{n,1}$, $v_h^{n,2}$, $v_h^{n+1} \in V_h$, such that 
\begin{subequations}
\label{eq:RK3}
\be 
  & & (v_h^{n,1} - v_h^n, \varphi_h) = \dt \mathcal{H}(v_h^n, \varphi_h), 
 \quad \forall \varphi_h \in V_h \\
  & & (v_h^{n,2} - \frac{3}{4} v_h^n - \frac{1}{4} v_h^{n,1}, \varphi_h)  
=\frac{\dt}{4} \mathcal{H}(v_h^{n,1}, \varphi_h), 
 \quad \forall \varphi_h \in V_h \\
 & & (v_h^{n+1} - \frac{1}{3} v_h^n - \frac{2}{3} v_h^{n,2}, \varphi_h) 
= \frac{2 \dt}{3} \mathcal{H}(v_h^{n,2}, \varphi_h), 
   \quad \forall \varphi_h \in V_h
\ee
\end{subequations}
where
$$
(\phi, \varphi)=\int_I \phi \,\varphi \,dx
$$
and
$$
{\mathcal H}_j(\phi_h, \varphi_h)= \int_{I_j} c \phi_h \,(\varphi_h)_x dx 
- c ( (\phi_h)^{-}_\jpud (\varphi_h)^{-}_\jpud -(\phi_h)^{-}_\jmud 
(\varphi_h)^{+}_\jmud),  
$$
$$
\mathcal H(\phi_h, \varphi_h) =\sum_j \mathcal H_j(\phi_h, \varphi_h).
$$
We shall denote this RKDG solver  by $v_h^{n+1}=G_{\dt}^{RK} (v_h^n)$.

\medskip

After introducing DG schemes for the linear transport problem, for 
the obstacle equation \eqref{eq:02}, we shall consider two approaches: 
one by using the $L^2$ projection
\be \label{eq:DGscheme}
  %u^{n+1} := \Pi_h \bigg(\max(u^n(\cdot - c\dt),\ \tilde g(\cdot))\bigg),
  u_h^{n+1} := \Pi_h \bigg(\max(G_\dt (u_h^{n}),\ \tilde g)\bigg),
\ee
where $G_\dt=G_{\dt}^{SL}$ or $G_{\dt}^{RK}$ and 
\be\label{eq:g1}
  \tilde g \equiv g_\dt  \quad \textrm{or} \quad \tilde g\equiv  
\max(g(x), g(x-c \dt)).
\ee
The idea of the formulation above is to try to follow the relation 
\eqref{eq:vn1exact}, and the projection step is 
to project the function into the piecewise polynomial space  $V_h$. 

Unfortunately, the scheme \eqref{eq:DGscheme} is difficult to 
implement, because we need to compute the maximum of two functions, 
which requires locating the roots of $G_\dt (u_h^{n})-\ \tilde g$.
Another more practical approach is to define $u^{n+1}_h$ as the 
unique polynomial in $V_h$ such that 
\be \label{eq:DGscheme2}
  \qquad  u_h^{n+1}(x^j_\ma) :=  \max\big(G_\dt (u_h^{n})(x^j_\ma),\ 
\tilde g(x^j_\ma)\big), \quad \forall \,j=1, \ldots, N,\, \ma=0, 
\ldots, k,
\ee
where $(x^j_\ma)_{\ma=0,\dots,k}$ are the $k+1$ Gauss-Legendre 
quadrature points on the interval $I_j$, and $w^j_\ma$ are the 
corresponding quadrature weights. Those schemes were studied in 
details and stability was established in \cite{bok-che-shu-14}. 
We shall see in later sections that definitions \eqref{eq:DGscheme} 
or \eqref{eq:DGscheme2} lead to similar
error estimates, although the second approach is much easier to implement.

Finally, we remark that in \cite{bok-che-shu-14}, forward Euler and 
TVD-RK2 temporal discretizations are also considered.  However,  
the stability restriction for the time step for the forward Euler 
method is rather severe as $\dt \leq C h^2$, and the stability 
proof of TVD-RK2 only works for piecewise linear polynomials. 
Therefore, in this paper, we will only consider TVD-RK3 time 
discretizations.

%\bigskip

\section{Preliminaries}
%----------------------
\label{sec:prelim}

In this section, we collect some lemmas which will be used in 
our convergence proof, and discuss properties of the obstacle 
solutions. Here and below, we use $C$ (possibly with subscripts) 
to denote a positive constant depending solely on the exact 
solution, which may have a different value in each occurrence.

Let us introduce, for $\ell\geq 0$, the following function sets:
\beno
 & &  \mathcal C^{\ell+1}_{p,L,c_0}(0,1):= \bigg\{ v:(0,1)\conv\R,\ 
\mbox{$v$ Lipschitz continuous with $\|v'\|_{L^\infty}\leq L$,}\\
 & &  \hspace*{3cm} \mbox{$v$ piecewise $C^{\ell+1}$ with 
$\|v^{(\ell+1)}\|_{L^\infty} \leq c_0$,}\\
 & &  \hspace*{3cm} \mbox{and $v$ admits at most $p\geq 0$ non 
regular points.}  \bigg\}
\eeno
where $v^{(\ell+1)}$ denotes the $(\ell+1)$-th derivative almost 
everywhere, and
\beno
 & &  \Delta^{2}_{q}(0,1):= \bigg\{ g:(0,1)\conv\R,\ \mbox{$g$ has 
at most $q$ local maxima points} \\
 & &  \hspace*{3cm} \mbox{and $g$ is twice differentiable at each 
local maxima.} \bigg\}
\eeno
The following $\ell^2$ pseudo-norm definition will also be used:
\be
   \| f \|_{\ell^2}:=\bigg( \sum_{i,\ma} w^i_\ma | f(x^i_\ma) |^2  
h_i\bigg )^{1/2}.
\ee
In particular, using the Gauss-Legendre quadrature rule, for any 
$f\in V_h$ we have  $\| f\|_{\ell^2} = \| f\|_{L^2}$. From this 
point on, we will use $\|f\|$ to denote $ \| f\|_{L^2}$, and 
$\|f\|_{D}$ to denote $\|f\|_{L^2(D)}$ for a given domain D. 

\subsection{Properties of projections and the obstacle function}
For the RKDG method, it is necessary to consider the following 
Legendre-Gauss-Radau projection $\pro$. For any function $\varphi$, $\pro 
\varphi \in V_h$, and for any element $I_j$, it holds that 
$$ 
(\pro \varphi)_{j+1/2}^-=\varphi_{j+1/2}^-, \qquad \int_{I_j} 
(\pro \varphi -\varphi) \, \psi_h dx=0, \quad \forall\, \psi_h 
\in P^{k-1}(I_j) .
$$

\medskip

In the lemma below, we will first establish the projection 
properties for functions in the space $C^{\ell+1}_{p,L,c_0}$.

\begin{lem}\label{lem:bound1}
Let $\ell\geq 0$ and
let $\varphi$ be in $C^{\ell+1}_{p,L,c_0}$: $\varphi$ is Lipschitz 
continuous, piecewise $C^{\ell+1}$,
with at most $p\geq 0$ non regular points.
Then there exists a constant $C\geq 0$, depending only on 
$\ell,p,L,c_0$ such that
$$
  \| \varphi - P_h \varphi \|
  \leq C h^{q_{k,\ell}},
$$
where the projection $P_h$ can be either $\pro$ or $\Pi_h$, 
and $q_{k,\ell}$ is defined as
\be\label{eq:qkell}
  q_{k,\ell}:= \min \left( \min(k,\ell)+1,\frac{3}{2} \right) \equiv
    \left\{ \barr{ll} 3/2 & \mbox{if $k,\ell\geq 1$} \\ 
                       1  & \mbox{if $k=0$ or $\ell=0$} \earr\right.
\ee 
\end{lem}
\begin{proof}
Using the property of the projections \cite{ciarlet1978finite}, 
on each regular cell $I_j$, we have
$$
   \|\varphi-P_h \varphi\|_{ I_j} \leq  C h^{\min(\ell, k)+1} 
\|\varphi\|_{H^{\ell+1}(I_j)}  \leq  C h^{\min(\ell, k)+1}  
$$
since $\varphi \in C^{l+1}$ on $I_j$.
Therefore,  
$$
  \| \varphi -P_h \varphi\|_{L^2(\cup I_j,\ \varphi_{|I_{j}} 
\mathrm{regular})}
  = \left(\int_{\cup I_j,\ \varphi_{|I_j} \mathrm{regular}} 
|\varphi - P_h \varphi|^2 dx\right)^{1/2}
  \leq C h^{\min(\ell, k)+1}.
$$

On the other hand, on the intervals $I_j$ where $\varphi$ is not 
regular, we have
$$
   \|\varphi-P_h \varphi\|_{ I_j} \leq  C h^{\min(0, k)+1} 
\|\varphi\|_{H^1(I_j)} \leq C h \|\varphi\|_{H^1(I_j)} \leq C h^{3/2}
$$
because $\varphi$ is Lipschitz continuous.
Hence summing up on the ``bad" intervals $I_j$ (with at most $p$ 
such intervals),
$$
   \|\varphi-P_h \varphi\|_{L^2(\cup_{bad} I)} \leq C p^{1/2}   h^{3/2}.
$$
Summing up the bounds with bad intervals and good ones, we prove the 
desired result.
\end{proof}

\begin{comment}
We also give a bound when $v$ is discontinuous.
\begin{lem}\label{lem:bound1-discont}
Assume $v$ is piecewise $C^{\ell+1}$. ($v$ can be discontinuous here).
Then we have, for some constant $C\geq 0$,
$$ \| v - \Pi_h v \|_{L^2} \leq C h^{1/2}.
$$
\end{lem}
\begin{proof}
Let $I=]x_\imud,x_\ipud[$, and $v_i=v(x_\imud^+)$. We have 
for all $x\in I$,
$|v(x)-v_i|\leq TV(v,]x_\imud,x]) \leq TV(v,I)$, where we have 
denoted $TV(v,Q)$ the total variation of $v$ on a given interval $Q$.
Hence
$$ \| v-\Pi_h v\|_{L^2(I)} \leq \| v - v_i \|_{L^2(I)} \leq 
TV(v,I)\ h^{1/2}. $$
Summing up the previous bounds on a set of $p$ ``bad" intervals 
$I$ where $v$ is no more that total variation bounded, we obtain:
$$ 
  \| v-\Pi_h v\|^2_{L^2(\cup_{{\textrm{bad }} I} I)} 
    \leq   \sum_{{\textrm{bad }} I} \| v - v_i \|^2 \leq  ph\ 
TV(v,I)^2 \leq C h.
$$
On any interval $I$ where $v$ is regular, we have the usual bound 
for $\| v-\Pi_h v \|_{L^\infty(I)}\leq C h^{\ell+1}$,
from which we can conclude to the desired bound.
\end{proof}
\end{comment}

We now state a similar estimate for the $\ell^2$ norm:

\begin{lem}\label{lem:bound3}
Assume that $\varphi$ belongs to $C^{\ell+1}_{p,L,c_0}$, $\ell\geq 0$, 
then we have
\be\label{eq:bound3}
  \qquad
  \| \varphi - P_h \varphi \|_{\ell^2} = \left(\sum_{j,\ma} \wja 
\big|v(\xja) - (P_h \varphi)(\xja) \big|^2 h_j\right)^{1/2} 
  \leq C h^\qkell
\ee
where $P_h=\pro$ or $\Pi_h$,  $\qkell$ is defined as in 
\eqref{eq:qkell}, and the constant $C$ depends only on $p,L$ and $c_0$.
\end{lem}
\begin{proof}
On each regular cell $I_j$, we have \cite{ciarlet1978finite},
$$
   \|\varphi-P_h \varphi\|_{L^\infty(I_j)} \leq  C h^{\min(\ell,k)+1/2} 
\|\varphi\|_{H^{\ell+1}(I_j)}  \leq C h^{\min(\ell,k)+1} .
$$

Then using the fact that $\wja \geq 0$, $\sum_\ma \wja = 1$ 
we obtain that
$$  
\left(\sum_{I_j,\ \varphi_{|I_j} \mbox{ regular} }\sum_{\ma} \wja 
\big|\varphi^n(\xja) - (P_h \varphi^n)(\xja) \big|^2 h_j\right)^{1/2}  
\leq C h^{\min(\ell,k)+1}.
$$
On the other hand, when the interval $I_j$ is such that $\varphi$ 
contains a non-regular point, we can write
$$
   \|\varphi-P_h \varphi\|_{L^\infty(I_j)} \leq  C h^{\min(0,k)+1/2} 
\|\varphi\|_{H^1(I_j)} \leq C h.
$$
Therefore
\beno
  &&  \left(\sum_{I_j,\ \varphi_{|I_j} \mbox{ not regular} }\sum_{\ma} 
\wja \big|\varphi^n(\xja) - (P_h \varphi^n)(\xja) 
\big|^2 h_j\right)^{1/2}\\
    & \leq & \left(\sum_{I_j,\ \varphi_{|I_j} \mbox{ not regular} } 
h (C h)^2\right)^{1/2}  \\ & \leq & C p^{1/2}   h^{3/2}.
\eeno
This concludes our proof.
\end{proof}

\bigskip

We now turn to some estimates related to the obstacle function $g_\dt$.

\begin{lem}\label{lem:gapprox}
Assume that $g\in \Delta^2_q(0,1)$ for some integer $q\geq 1$.
Let $\tilde g(x):=\max(g(x),g(x-c\dt))$.
There exists a constant $C\geq 0$ such that 
\be\label{eq:g_estimate2a}
   \| \tilde g   - g_\dt \| \leq C \sqrt{q} \dt^{5/2}.
\ee
and
\be\label{eq:g_estimate2b}
   \| \tilde g   - g_\dt \|_{\ell^2} \leq C \sqrt{q}\ 
\sqrt{\dt+h}\,\dt^{2}.
\ee
\end{lem}
\begin{proof}
Denote $\cM_g$ the set of local maximum points of $g$, if 
$[x-c\dt,x]\cap \cM_g=\emptyset$, then
we see that $\tilde g(x)-g_\dt(x)=0$. Furthermore
\be
  \int_{ \{x,\ g_\dt(x)\neq \tilde g(x) \} } dx 
   & \leq & \int_{ \{x,\ [x-c\dt,x]\cap \cM_g\neq \emptyset\} } dx 
   \ \leq \ q \, c\dt
   \label{eq:oo2}
\ee
since there are at most $q$ local maxima.

Consider the case when $[x-c\dt,x]$ contains at least a local maxima  
of $g$. Assuming that $\dt$ is small enough we can assume that 
$x^*$ is the only local maximum of $g$ on the interval $[x-c\dt,x]$, 
so that $g_\dt(x)=g(x^*)$. Then, 
\be\label{eq:oo1}
  |g(x)-g_\dt(x)|=|g(x)-g(x^*)| \leq C |x-x^*|^2 \leq C \dt^2,
\ee
since $g$ is twice differentiable at $x^*$.

 Combining \eqref{eq:oo1} and \eqref{eq:oo2} we obtain the 
bound \eqref{eq:g_estimate2a}.

For the second estimate \eqref{eq:g_estimate2b}, we make use of a 
minimal covering  $\cup I$ of the set 
$$ 
  \bigg\{ x,\ [x-c\dt,x]\cap \cM_g\neq \emptyset\bigg\},
$$
using mesh intervals. The length of this covering is bounded by  
$c\dt+2h$ for each maximum point, since in order to cover
any interval $[a,b]$  we may need two more mesh intervals $I$, of 
length $\leq 2h$ than the minimum required length $b-a$. Overall 
the length of the total covering is bounded by $q(c\dt + 2h)$, hence we 
obtain the desired result.

\end{proof}

\subsection{Properties of the obstacle solutions} 
We shall impose some restrictions on the regularity of the obstacle 
solutions as described below.

\begin{defi}{\bf[``no shattering" property]}
\label{def:noshat}
For a given $T\geq 0$, we will say that the exact solution $u$ of the problem \eqref{eq:02} is 
``not shattering"
if there exists some $\ell\geq 0$ and constants 
$p,L,c_0$ such that the exact solution satisfies, for any $t\in[0,T]$, 
$u(t,.)\in C^{\ell+1}_{p,L,c_0}$.
\end{defi}

Recall that the exact solution satisfies 
$u(t,x)=\max(u_0(x-ct),g_t(x))= u_0(x-ct) + \max(0,g_t(x)-u_0(x-ct))$.
Therefore if $u$ is not shattering, it implies that  
$u_0(x-ct)-g_t(x)$ has bounded number of zeros
(since otherwise $x\conv u(t,x)$ would have an unbounded number of 
singularities).

A typical example where shattering occurs 
(therefore not satisfying definition \ref{def:noshat}), can be 
constructed as follows.
Suppose the domain $\Omega$ contains the interval $(-2,2)$,  let
$u_0(x)\equiv (x-1) +  (x-1)^3 \sin(1/(x-1))$ and $g(x)\equiv x$ 
on $[-2,2]$, together with velocity constant $c=1$.
The function $u_0$ is of class $C^2$ on the interval $(-2,2)$.
Notice then, $g_t(x)=\max_{\mt\in[0,t]} g(x-\mt) = g(x)$ for 
$t\in[0,1]$ and $x\in[-1,1]$, and
the exact solution is therefore $u(t,x)=\max(u_0(x-t),g(x))$ for 
$t\in [0,1]$ and $x\in[-1,1]$.
So at time $t=1$, 
$u(1,x)=\max(u_0(x-1),g(x)) \equiv x + \max(x^3 \sin(1/x),0)$ 
(for $x \in [-1,1]$).
This function has an infinite number of non regular points in the 
interval $[-1,1]$, and therefore does not satisfy
the ``no shattering" property at time $t=1$.

It is not easy to state precise conditions on the initial data 
$u_0$ and $g$ to ensure that the no shattering property 
will be satisfied. Mainly, $\forall t$, $x\conv g_t(x) - u_0(x-ct)$ 
should have only a finitely bounded number of zeros,
as is detailed below.
However, it is clear that shattering 
will not occur for generic data $u_0$ and $g$.
This definition still allows for a finite (bounded) number of 
singularities in $u(t_n,.)$,
as is generally the case when taking the maximum of two 
regular functions.
Finally we also give an example (see Lemma~\ref{lem:noshat2} below) 
where we can prove that the ``no shattering" condition is satisfied.

\begin{lem}\label{lem:noshat}
Assume that, for a given $T>0$,
\begin{subequations}
\be 
  & & g\in C^{\ell+1}_{p_g,L_g,c_g}(0,1),                \label{ass1} \\
  & & u_0\in C^{\ell+1}_{p_{u_0},L_{u_0},c_{u_0}}(0,1),  \label{ass2} 
\ee 
and that, $\forall \,t\in [0,T]$, 
\be  \label{ass3}
  & & x\conv g_t(x) - u_0(x-ct) \quad \mbox{has a finitely bounded number of zeros in $(0,1)$,} 
\ee
\end{subequations}
the bound being independent of $t$.

Then there exists constants $p,L,c_0$ such that, $\forall t\in[0,T]$, 
$u(t,x)$ belongs to $C^{\ell+1}_{p,L,c_0}$ 
(with $L=\max(L_g,L_{u_0})$, $c_0=\max(c_g,c_{u_0})$,
and $p$ that are independent of $t$.)
\end{lem}
\begin{proof}
First, for $g\in C^{\ell+1}_{p_g,L_g,c_g}(0,1)$, 
if we denote $M$  (resp. $m$) to be 
the number of local maxima (resp. minima) of $g$, then we have 
$g_t\in C^{\ell+1}_{p_g + 2M + m,L_g,c_g}(0,1)$. This is because 
the Lipschitz constant of $g_t$ is bounded by $L_g$ by an elementary
verification.
Each local maxima of $g$ may develop into two singularities in $g_t$,
each local minima may develop into one singularity in $g_t(\cdot)$, 
and each singular point of $g$ may continue to be a
singularity in $g_t(\cdot)$.
Hence the total number of non-regular points in $g_t(\cdot)$ will be 
bounded by $2M + m + p_g$. The bound of the $(\ell+1)$
derivative can also be obtained easily.

 Because the exact solution is given by $u(t,x)=\max(u_0(x-ct), 
g_t(x)) = u_0(x-ct) + \max(0,g_t(x) - u_0(x-ct))$,
the Lipschitz constant of $u(t,.)$ is therefore bounded by 
$\max(L_{u_0(\cdot-ct)},L_g)=\max(L_{u_0},L_g)$.

On the other hand, the number of singular points of $\max(0,g_t(x) 
- u_0(x-ct))$ is bounded by the sum of the number
of singular points of the function $x\conv g_t(x) - u_0(x-ct))$, 
plus the number of zeros 
of the same function, which is assumed to be bounded independently 
of $t\geq 0$.
Hence the the number of singular points of $x\conv u(t,x)$ is 
bounded independently of $t\geq 0$.

Finally the bound on the $x$-partial derivative 
$\|u^{(\ell+1)}(t,\cdot)\|_{L^\infty}$, in the regular region of $u$, 
is easily obtained, because $u$ is then locally one of the two 
functions $g_t$ or $u_0(\cdot-ct)$.
\end{proof}

\begin{lem}\label{lem:noshat2}
Assume that $g$ and $u_0$ satisfy the regularity assumptions~\eqref{ass1} and~\eqref{ass2} of Lemma~\ref{lem:noshat}.
%\be 
%  & & g\in C^{\ell+1}_{p_g,L_g,c_g}(0,1),                \label{ass1} \\
%  & & u_0\in C^{\ell+1}_{p_{u_0},L_{u_0},c_{u_0}}(0,1),  \label{ass2} 
%\ee 
Assume furthermore that there exists a constant $L_1\geq 0$ such that $|u'_0(x)| \geq L_1 > L_g$ for a.e. $x\in(0,1)$.
Then \eqref{ass3} holds and $u$ satisfies the ``no shattering" assumption.
\end{lem}
\proof 
As we can see in Lemma~\ref{lem:noshat}, $g_t$ also has a Lipschitz 
constant $\leq L_g$. On the other hand on each regular part of $u_0$ 
there is a slope 
$\geq L_1$ which is strictly greater that $L_g$. By an elementary 
verification, one can show that assumption \eqref{ass3} is satisfied and 
that the result of Lemma~\ref{lem:noshat} can be applied.
\endproof

%--------------------------------------------------
\section{Convergence of the SLDG schemes}
%--------------------------------------------------
\label{sec:sldg}
%----------------------------------------------

In this section, we will provide the convergence proof of the 
SLDG scheme. In particular, we will proceed in three steps. First, 
we will establish error estimates of the SLDG methods for the 
linear transport equation
$$
v_t +c\, v_x=0, \qquad v(0,x)=v_0(x).
$$
We will then generalize the results to scheme \eqref{eq:DGscheme}, 
and finally to scheme \eqref{eq:DGscheme2} for the obstacle problem.

\subsection{Convergence of the SLDG scheme for the linear 
advection equation}
%---------------------------------------------------------

We first consider the linear equation  $v_t+ c v_x=0$, for which
$$ v(t+\dt)= v(t,x-c\dt). $$
We denote $v^n(\cdot)=v(t^n,\cdot)$, and we define the numerical 
solution of the SLDG method at $t^n$ to be $v_h^n$.
In particular, the scheme writes: initialize with 
$ v_h^0:=\Pi_h v_0$, and 
$v_h^{n+1}=G_{\dt}^{SL} (v_h^n)= \Pi_h(v_h^n (\cdot - c\dt))$ for 
$n\geq 0$.

\begin{thm}\label{th:1} We consider  $v_t + c v_x=0$, $v(0,x)=v_0(x)$.
If $v_0 \in  \mathcal C^{\ell+1}_{p,L,c_0}(0,1)$, then we have for 
all $n$ such that $n\dt \leq T$,
\beno
  \|v_h^{n} - v^{n}\|  \leq  CT\ \frac{h^\qkell}{\dt}
\eeno
for some constant $C\geq 0$ independent of $n$,
and $\qkell$ is defined in \eqref{eq:qkell}.
\end{thm}

\begin{proof}
If $v_0 \in  \mathcal C^{\ell+1}_{p,L,c_0}(0,1)$, then $v^n \in  
\mathcal C^{\ell+1}_{p,L,c_0}(0,1)$, and due to Lemma 
\ref{lem:bound1}, we have:
$$ 
\|v^{n}(\cdot - c\dt) - \Pi_h(v^n(\cdot - c\dt)) \| \leq    C h^\qkell,
$$
for some constant $C$ independent of $n$.
Hence
\beno
 && \|v_h^{n+1} - v^{n+1}\|
    =   \|\Pi_h(v_h^n(\cdot - c\dt) - v^{n}(\cdot - c\dt)\| \\
   & &\leq  \|\Pi_h(v_h^n(\cdot - c\dt) - \Pi_h(v^n(\cdot - c\dt)\|
+ \|\Pi_h(v^n(\cdot - c\dt)- v^{n}(\cdot - c\dt)\|\\
   & &\leq  \|\Pi_h(v_h^n(\cdot - c\dt) - \Pi_h(v^n(\cdot - c\dt)\| 
+ C h^\qkell \\
   & &\leq  \| v_h^n(\cdot - c\dt) - v^n(\cdot - c\dt)\| + C h^\qkell \\
   & &\leq  \| v_h^n - v^n\| + C h^\qkell,
\eeno
where we have used the fact $\| \Pi_h u  \| \leq \| u\| $ in the 
fourth row, and the periodic boundary conditions in the last row.
As for the initial condition, we have
$$ 
\| v_h^0 - v_0 \| = \| \Pi_h v_0 - v_0 \| \leq Ch^\qkell .
$$
Finally we obtain for any given $T\geq 0$ the existence of a 
constant $C\geq 0$ (independent of $T$ and of $n$), such that
\beno
  \|v_h^{n} - v^{n}\|   \leq  C(n+1) \,h^\qkell
  \leq CT\ \frac{h^\qkell}{\dt},
\eeno
for all $n$ such that $n\dt \leq T$.
\end{proof}

\medskip

Therefore, if $\min(k,\ell)=0$, then $\|v_h^{n} - v^{n}\| \leq 
C T \frac{h}{\dt}$, and we need $h=o(\dt)$ for the convergence.
If otherwise $\min(k,\ell)\geq 1$, it holds  $\|v_h^{n} - v^{n}\|   
\leq C T \frac{h^{3/2}}{\dt}$ and we would only
need $h=o(\dt^{2/3})$ for the convergence.

\subsection{Convergence of the first SLDG scheme in the obstacle case}
%---------------------------------------------------

Now we turn to scheme \eqref{eq:DGscheme} for the nonlinear equation 
\eqref{eq:02}. 
In particular, the scheme writes: initialize with 
$ u_h^0:=\Pi_h u_0$, and 
$u_h^{n+1} = \Pi_h (\max(G_{\dt}^{SL} (u_h^n) , \tilde g))$ for $n \geq 0$.

\begin{thm}\label{th:2}
Assume that the exact solution $u$ is not shattering in the sense 
of Definition~\ref{def:noshat}, for some integer $\ell\geq 0$.
Then the following error bound holds:
\beno
  \|u_h^{n} - u^{n}\|   \leq  CT\ \frac{h^\qkell}{\dt} + CT\,\frac{\|\tilde g - g_\dt\|}{\dt}
\eeno
for some constant $C\geq 0$ independent of $n$.
In particular,

$(i)$ If $\tilde g= g_\dt$, then $\forall \,t_n\leq T$:
\beno
  \|u_h^{n} - u^{n}\|   \leq  CT\ \frac{h^\qkell}{\dt}
\eeno

$(ii)$ If $\tilde g(x):=\max\big( g(x), g(x-c\dt) \big)$, and if 
$g \in \Delta^2_q(0,1)$, then $\forall \,t_n\leq T$:
\beno
  \|u_h^{n} - u^{n}\|   \leq  CT\ \frac{h^\qkell}{\dt} + CT\, \dt^{3/2}.
\eeno
\end{thm}

\begin{proof}
By Lemma~\ref{lem:bound1} and the ``no shattering" assumption,
\be \label{eq:b0}
  \|\Pi_h u^{n} - u^{n}\|  \leq C h^\qkell .
\ee
Hence, from the definitions, 
\beno
  & &\|u_h^{n+1} - u^{n+1}\| 
    \leq  \|u_h^{n+1} - \Pi_h u^{n+1}\|  + \|\Pi_h u^{n+1} - u^{n+1}\|  \\
  & &\leq  \|\Pi_h( \max(G_{\dt}^{SL} (u_h^n), \tilde g) ) - 
\Pi_h( \max(u^{n}(\cdot - c\dt), g_\dt)) \|    + C h^\qkell\\
  & &\leq  \| \max(G_{\dt}^{SL} (u_h^n), g_\dt)  - \max(u^{n}(\cdot 
- c\dt),\tilde g) \|    + C h^\qkell.
\eeno
Using the fact that
\beno
  |\max(a_1,b_1)- \max(a_2,b_2)|\leq |a_1-a_2| + |b_1-b_2|,
\eeno
we obtain
\be
 && \hspace{-1cm} \|u_h^{n+1} - u^{n+1}\| 
    \ \leq \ \| G_{\dt}^{SL} (u_h^n) - u^n(\cdot - c\dt)\| 
+ \| \tilde g -  g_\dt \|  +  C h^\qkell  \nonumber\\
    && 
    \leq  \| \Pi_h(u_h^n (\cdot - c\dt))- \Pi_h (u^n(\cdot - c\dt))\| 
+ \| \tilde g -  g_\dt \|  +  C h^\qkell \nonumber \\
   & &\leq  \| u_h^n- u^n\|  + \| \tilde g -  g_\dt \|  +  C h^\qkell  
\label{eq:almost_final_bound}
\ee
where we have used again \eqref{eq:b0} and the periodic boundary 
condition. Using Lemma~\ref{lem:gapprox} and by induction on $n$, 
we are done.
\end{proof}

\begin{rem}\label{rem:SLDG1}
Defining $\tilde g$ as in $(ii)$, and assuming $\min(\ell,k)\geq 1$,
the error is bounded by $O(\frac{h^{3/2}}{\dt}) + O(\dt^{3/2})$.
Therefore the optimal estimate is obtained when $h^{3/2} \equiv 
\dt^{5/2}$, or $\dt\equiv h^{3/5}$, and the error is of order $O(h^{9/10})$.
\end{rem}

\subsection{Convergence of the SLDG scheme defined with Gauss-Legendre quadrature points}
%---------------------------------------------------
Now we turn to scheme \eqref{eq:DGscheme2} for the nonlinear equation 
\eqref{eq:02}. 
In particular, the scheme writes: initialize with 
$ u_h^0:=\Pi_h u_0$, and 
 $u_h^{n+1}$ is  defined as the unique polynomial in $V_h$ such that:
\be\label{eq:sldg_obs2}
   u^{n+1}_h(\xja):= \max(G_{\dt}^{SL} (u_h^{n})(x^j_\ma),\ 
\tilde g(x^j_\ma)), \quad \forall j, \,\ma.
\ee

\begin{thm}\label{th:3}
Let $\ell\geq 0$ and assume that the exact solution is not shattering 
in the sense of Definition~\ref{def:noshat}.
The following error bound holds:
\beno
  \|u_h^{n} - u^{n}\|   \leq  CT\ \frac{h^\qkell}{\dt} + CT\,\frac{\|\tilde g - g_\dt\|}{\dt},
\eeno
for some constant $C\geq 0$ independent of $n$. In particular,

$(i)$ If $\tilde g= g_\dt$, then $\forall t_n\leq T$:
\beno
  \|u_h^{n} - u^{n}\|  \leq  CT\ \frac{ h^\qkell}{\dt}.
\eeno
%for some constant $C\geq 0$ independent of $n$.\\

$(ii)$ If $\tilde g(x):=\max\big( g(x), g(x-c\dt) \big)$ and $g \in 
\Delta^2_q(0,1)$, then $\forall t_n\leq T$:
\beno
 \|u_h^{n} - u^{n}\|   \leq  CT\ \frac{h^\qkell}{\dt} + CT\, \dt \sqrt{\dt+h}.
\eeno
%for some constant $C\geq 0$ independent of $n$. 
\end{thm}

\begin{proof}
In view of Lemma~\ref{lem:bound1} and the ``no shattering" property,  
we have
$$
  \| u^{n+1}- \Pi_h u^{n+1} \|  \leq  C  h^\qkell.
$$

Now we turn back to the error estimate, and consider the case of 
$\tilde g= g_\dt$:
\beno
  &&\| u_h^{n+1} - u^{n+1} \| 
   \leq  \| u_h^{n+1} - \Pi_h u^{n+1} \|  + C  h^\qkell \\
&&   =\| u_h^{n+1} - \Pi_h u^{n+1} \|_{\ell^2} + C  h^\qkell \\
&&   \leq \| u_h^{n+1} -  u^{n+1} \|_{\ell^2} + C  h^\qkell \\
  & &\leq   \left(\sum_{j,\ma} w^j_\ma  \left| u_h^{n+1}(\xja) 
- u^{n+1}(x^j_a) \right|^2 h_j \right)^{1/2} + C  h^\qkell
  \eeno
where in the third line we have used~\eqref{eq:bound3} for $u^{n+1}$.
Because of the DPP, for all points $x$, the exact solution satisfies:
$$
   u^{n+1}(x)=\max(u^n( x - c\dt),\ g_\dt(x)).
$$
Therefore, for $\forall j, \,\ma$:
\beno 
  \label{eq:bunvn}
  \bigg|(u_h^{n+1}-u^{n+1})(x^j_\ma)\bigg| \leq
  \big|G_{\dt}^{SL} (u_h^{n})(x^j_\ma) -  u^n( x^j_\ma - c\dt) \big| 
+ \big|\tilde g(x^j_\ma) - g_\dt(x^j_\ma) \big| 
\eeno
and
\beno
  \| u_h^{n+1} - u^{n+1} \| 
  & \leq & \left(\sum_{j,\ma} w^j_\ma  \big|G_{\dt}^{SL} (u_h^{n})
(x^j_\ma)-  u^n( x^j_\ma - c\dt) \big|^2 h_j \right)^{1/2} \\
  &      & 
     \hspace{2cm} + \left(\sum_{j,\ma} w^j_\ma  \big|\tilde g(x^j_\ma) 
- g_\dt(x^j_\ma) \big|^2 h_j \right)^{1/2} + C  h^\qkell \\
  & \leq & \|G_{\dt}^{SL} (u_h^{n})-u^n(\cdot - c \dt)\|_{\ell^2} 
+\|\tilde g -g_\dt\|_{\ell^2} + C  h^\qkell \\
  & \leq & \|G_{\dt}^{SL} (u_h^{n})-\Pi_h u^n(\cdot - c \dt)\|_{\ell^2} 
+\|\tilde g -g_\dt\|_{\ell^2} + C  h^\qkell \\
  & =    & \|\Pi_h u_h^n(\cdot -c\dt) -\Pi_h u^n(\cdot - c 
\dt)\|_{\ell^2} +\|\tilde g -g_\dt\|_{\ell^2} + C  h^\qkell \\
  & =    & \|\Pi_h u_h^n(\cdot -c\dt) -\Pi_h u^n(\cdot - c \dt)\|  
+\|\tilde g -g_\dt\|_{\ell^2} + C  h^\qkell \\
  & \leq & \|u_h^n(\cdot -c\dt) - u^n(\cdot - c \dt)\|  +\|\tilde g 
-g_\dt\|_{l^2} + C  h^\qkell \\
  & \leq & \|u_h^n- u^n\|  +\|\tilde g -g_\dt\|_{\ell^2} + C  h^\qkell
\eeno
where in the fourth line we have used Lemma~\ref{lem:bound3} for the 
function $u^n(\cdot-c\dt)$.  
Using Lemma~\ref{lem:gapprox} and by induction on $n$, we are done.
\end{proof}

\begin{rem}\label{rem:SLDG2}
Defining $\tilde g$ as in $(ii)$, and assuming $\min(\ell,k)\geq 1$,
the error is bounded by $O(\frac{h^{3/2}}{\dt}) + O(\dt \sqrt{\dt+ h}) $.
Therefore the optimal estimate is obtained when $h^{3/2} \equiv \dt^2 \sqrt{\dt+h}$. 
So $\frac{h}{\dt} \conv 0$, $h^{3/2} \equiv \dt^{5/2}$, or $\dt \equiv h^{3/5}$ (as in Remark~\ref{rem:SLDG1}),
and the error of the scheme defined with Gauss-Legendre quadrature points is again of order $O(h^{9/10})$ for this particular time stepping.
\end{rem}

\section{Convergence of RKDG schemes}
\label{sec:rkdg}
%------------------------------------------------

In this section, we will prove convergence for the RKDG schemes.
 We will proceed in three steps similar to the previous section.

\medskip

Firstly, let us recall the following properties for the bilinear 
operator $\mathcal H$.

\begin{lem}\label{lem:h} \cite{Zhang_Shu_RK3_2010}
For any $\phi_h, \varphi_h \in V_h$, we have
\beno
&& \mathcal H(\phi_h, \varphi_h) +\mathcal H(\varphi_h, \phi_h) = -\sum_j c [\phi_h]_{j+1/2} \cdot [\varphi_h]_{j+1/2}\\
&& \mathcal H(\phi_h, \phi_h)=-\frac{1}{2} \sum_j c [\phi_h]^2_{j+1/2}. 
\eeno
\end{lem}

We also recall inverse inequalities \cite{ciarlet1978finite} for the finite element space $V_h$. 
In particular, there exists a constant C (independent of $h$), such that,
for any $\varphi_h \in V_h$, 
$$\|(\varphi_h)_x\| \leq C h^{-1} \|\varphi_h\|, \qquad \|\varphi_h\|_{L^\infty} \leq C h^{-1/2} \|\varphi_h\|.$$

\subsection{Convergence of the RKDG scheme for the linear advection equation}
%---------------------------------------------------------
We first consider the linear equation  $v_t+ c v_x=0$, for which
%$$ u(t+\dt)= u(t,y_x(-\dt))=u(t,x-c\dt) $$
$$ v(t+\dt)= v(t,x-c\dt). $$
We still denote $v^n(\cdot)=v(t^n,\cdot)$. In particular, the scheme writes: initialize with 
$ v_h^0:=\Pi_h v_0$, and 
$v_h^{n+1}=G_{\dt}^{RK} (v_h^n)$ for $n\geq 0$.

This convergence proof closely follows the work in \cite{Zhang_Shu_RK3_2010} for smooth solution, 
but   additional difficulties are encountered because we consider solutions with less regularity. 
The main technique is to introduce piecewisely defined intermediate 
stage functions and the careful treatment of intervals containing 
irregular points.

\begin{thm}\label{th:rkdg11} 
We consider  $v_t + c v_x=0$, $v(0,x)=v_0(x)$.
Let $v_0$ be in $\mathcal C^{\ell+1}_{p,L,c_0}(0,1)$, $\ell\geq 2$, $k \geq 1$, and assume the 
CFL condition 
$$\dt \leq C_0 h$$
for $C_0$ small enough (the usual CFL condition for stability of the RKDG scheme).
The following bound holds:
\beno
  \|v_h^{n} - v^{n}\|   \leq  C_1 (h  + \frac{h^3}{\dt^2})^{1/2},
\eeno
for some constant $C_1\geq 0$ independent of $h, \dt, v_h$.

In particular, if $\dt/h$ is bounded from below ($\frac{\dt}{h} \geq 
\bar C_0$ for some constant $\bar C_0 >0$), then 
\beno
  \|v_h^{n} - v^{n}\|   \leq  C_1 h^{1/2}.
\eeno
\end{thm}

\begin{proof} We need to introduce some intermediate stages of the exact solution. 
%We always assume that the irregular points do not coincide with element boundaries, and this is achievable because the number of irregular points is finite. 
 Firstly we define
\beno
  & & v^{(1)}:=v^n - c \dt\  (v^n)_{x}
\eeno
where the spatial derivative $(v^n)_{x}$ should be understood in the weak sense. We notice that $v^{(1)}$ may become discontinuous at the irregular points of $v^n$.

To define the second intermediate stage $v^{(2)}$, we need to distinguish the ``good" and ``bad" intervals. 
Since $v^n \in  \mathcal C^{\ell+1}_{p,L,c_0}(0,1)$, when the mesh is fine enough, 
there are at most $p$ irregular intervals. Because of the CFL condition (which we assume implies in particular that $c \triangle t \leq h$), each irregular point at $t^n$ may influence at most three intervals at time $t^{n+1}$. Now we introduce sets 
%$\cB^n, Bi^n$, 
\newcommand{\cJ}{\mathcal I}
$\cB^n$ and $\cJ^n$ such that
$$ 
  %B^n= \bigcup_j I_j, \textrm{s.t.}\, I_j \,\textrm{or its immediate neighbors contain at least an irregular point of}\, v^n.
  \cB^n:= \bigcup_j I_j, \textrm{s.t.}\, I_j \,\textrm{or its immediate neighbors contain an irregular point of}\, v^n
$$ 
and the corresponding set of indices:
$$ 
  %Bi^n= \bigcup_j j, \textrm{s.t.}\, I_j \,\textrm{or its immediate neighbors contain at least an irregular point of}\, v^n.
  \cJ^n:= \bigcup_j j, \textrm{s.t.}\, I_j \,\textrm{or its immediate neighbors contain an irregular point of}\, v^n.
$$ 
Therefore  $meas(\cB^n) \leq 3 p h$ and $Card(\cJ^n)\leq 3p$. (In the case of the irregular points located exactly at the cell interface $x_{j+1/2}$, we include the point's neighboring cells $I_j$, $I_{j+1}$ in $\cB^n$, and $j$, $j+1$ in $\cJ^n$.)

Now,   we define
\be
 & & \tilde{v}^{(2)}:=\left\{
  \begin{array}{ll}
    \frac{3}{4}v^n +\frac{1}{4} v^{(1)} - c \frac{\dt}{4}\ (v^{(1)})_x,  & \mbox{if } x \notin \cB^n,\\
    %\frac{3}{4}v^n +\frac{1}{4} v^{(1)}- \frac{\dt}{4} c (v^{n})_x, & x \in B^n \\
    %= v^n  - \frac{\dt}{2} c (v^{n})_x&
    \frac{3}{4}v^n +\frac{1}{4} v^{(1)} - c \frac{\dt}{4}\ (v^{n})_x  \equiv  v^n  - c \frac{\dt}{2}\ (v^{n})_x,  & \mbox{if } x \in \cB^n, 
  \end{array} \right.
\ee
For points not located in $\cB^n$, the definition coincides with \cite{Zhang_Shu_RK3_2010}. For points in $\cB^n$, $v^n$ is used instead of $v^{(1)}$ to avoid discontinuity at the irregular points.
Notice that this causes $\tilde{v}^{(2)}$ to be discontinuous at $\partial \cB^n$. For example, if $x_a \in \partial \cB^n$, then the jump of $\tilde{v}^{(2)}$  at $x_a$ is of magnitude $c \frac{\dt}{4}\ (v^{(1)}-v^n)_x(x_a)$, and this is bounded by $C L \dt^2$. By these arguments, we could add a linear interpolating function defined by $\frac{1}{4} L_a(x)$ which is nonzero only on $\cB^n$   to enforce continuity at $\partial \cB^n$, and $||L_a||_\infty<C \dt^2$, i.e. we introduce

\be
 & & v^{(2)}:=\left\{
  \begin{array}{ll}
    \frac{3}{4}v^n +\frac{1}{4} v^{(1)} - c \frac{\dt}{4}\ (v^{(1)})_x,  & \mbox{if } x \notin \cB^n,\\
    %\frac{3}{4}v^n +\frac{1}{4} v^{(1)}- \frac{\dt}{4} c (v^{n})_x, & x \in B^n \\
    %= v^n  - \frac{\dt}{2} c (v^{n})_x&
    v^n  - c \frac{\dt}{2}\ (v^{n})_x+\frac{1}{4}L_a(x),  & \mbox{if } x \in \cB^n, 
  \end{array} \right.
\ee
and $L_a$ is chosen to be a linear polynomial so that $v^{(2)}$ is continuous at $\partial \cB^n$.  We can now define

\be 
 & & v^{(3)}=\left\{
  \begin{array}{l l}
    \frac{1}{3}v^n +\frac{2}{3} v^{(2)}- c \frac{2\dt}{3}\ (v^{(2)})_x, & \mbox{if } x \notin \cB^n,\\
    %\frac{1}{3}v^n +\frac{2}{3} v^{(2)}- \frac{2\dt}{3} c (v^{n})_x, & x \in B^n \\
    %= v^n  - \dt c (v^{n})_x=v^{(1)}.&
    \frac{1}{3}v^n +\frac{2}{3} v^{(2)}- c \frac{2\dt}{3}\ (v^{n})_x \equiv  v^n  - c \dt\ (v^{n})_x +\frac{1}{6} L_a(x), & \mbox{if } x \in \cB^n.
  \end{array} \right.
\ee
Notice that for $x \notin \cB^n$, the definition is still consistent with  \cite{Zhang_Shu_RK3_2010} for smooth solutions, and it is well defined because 
$\ell\geq 2$. However, for   irregular intervals, the definition is modified due to the lower regularity of the solution. 

\smallskip

 Now we are ready to define the errors
\be
\label{eqn:decompose}
&&e^{(1)}:=v^{(1)}-v_h^{n,1}, \quad \xi^{(1)}:=\pro v^{(1)}-v_h^{n,1}, \quad \eta^{(1)}:=\pro v^{(1)}-v^{(1)},\notag\\
&&e^{(2)}:=v^{(2)}-v_h^{n,2}, \quad \xi^{(2)}:=\pro v^{(2)}-v_h^{n,2}, \quad \eta^{(2)}:=\pro v^{(2)}-v^{(2)},\\
&&e^{n}:=v^{n}-v_h^{n}, \quad \xi^{n}:=\pro v^{n}-v_h^{n}, \quad \eta^{n}:=\pro v^{n}-v^{n}.\notag
\ee
Clearly, 
\beno
e^{(1)}=\xi^{(1)}-\eta^{(1)}, \quad e^{(2)}=\xi^{(2)}-\eta^{(2)}, \quad e^{n}=\xi^{n}-\eta^{n}.
\eeno
Our next step is to establish the error equations.
First, let us recall that the numerical solution satisfies:
\beno
\int_{I_j} v_h^{n,1} \varphi_h dx &=& \int_{I_j} v_h^n \varphi_h dx + \dt \mathcal H_j (v_h^n, \varphi_h), \quad \forall \varphi_h \in V_h\\
\int_{I_j} v_h^{n,2} \varphi_h dx &=& \frac{3}{4} \int_{I_j} v_h^n \varphi_h dx+ \frac{1}{4} \int_{I_j} v_h^{n,1} \varphi_h dx  +\frac{\dt}{4} \mathcal H_j (v_h^{n,1}, \varphi_h), \quad \forall \varphi_h \in V_h\\
\int_{I_j} v_h^{n+1} \varphi_h dx &=& \frac{1}{3} \int_{I_j} v_h^n \varphi_h dx+ \frac{2}{3} \int_{I_j} v_h^{n,2} \varphi_h dx  +\frac{2\dt}{3} \mathcal H_j (v_h^{n,2}, \varphi_h). \quad \forall \varphi_h \in V_h.
\eeno
 From the definitions of $v^{(1)}, v^{(2)}, v^{(3)}$, we can verify
\beno
\int_{I_j} v^{(1)} \varphi_h dx &=& \int_{I_j} v^n \varphi_h dx + \dt \mathcal H_j (v^n, \varphi_h), \quad \forall \varphi_h \in V_h\\
\int_{I_j} v^{(2)} \varphi_h dx &=&\frac{3}{4} \int_{I_j} v^n \varphi_h dx+ \frac{1}{4} \int_{I_j} v^{(1)} \varphi_h dx  +\frac{\dt}{4} \mathcal H_j (v^{(\star1)}, \varphi_h), \quad \forall \varphi_h \in V_h\\
&&+\left\{
  \begin{array}{l l}
    0,& j \notin \cJ^n\\
    \frac{1}{4}(L_a,  \varphi_h), & j \in \cJ^n
  \end{array} \right.\\
\int_{I_j} v^{(3)} \varphi_h dx &=& \frac{1}{3} \int_{I_j} v^n \varphi_h dx+ \frac{2}{3} \int_{I_j} v^{(2)} \varphi_h dx  +\frac{2\dt}{3} \mathcal H_j (v^{(\star2)}, \varphi_h), \quad \forall \varphi_h \in V_h
\eeno
where for $j \in \cJ^n, (\star1)=n, (\star2)=n;$ otherwise, $(\star1)=(1), (\star2)=(2).$ Notice that the formulations above are correct because we have enforced continuity of the first function appearing
in operator $\mathcal H_j$ in all cases. 
In particular, the procedure to enforce continuity of $v^{(2)}$ at $\partial \cB^n$ turns out to be necessary here.
Combining the previous two relations, we derive the error equations
\beno
\int_{I_j} e^{(1)} \varphi_h dx &=& \int_{I_j} e^n \varphi_h dx + \dt \mathcal H_j (e^n, \varphi_h), \quad \forall \varphi_h \in V_h\\
\int_{I_j} e^{(2)} \varphi_h dx &=& \frac{3}{4} \int_{I_j} e^n \varphi_h dx+ \frac{1}{4} \int_{I_j} e^{(1)} \varphi_h dx  +\frac{\dt}{4} \mathcal H_j (e^{(1)}, \varphi_h), \quad \forall \varphi_h \in V_h\\
&&+\left\{
  \begin{array}{l l}
    0,& j \notin \cJ^n\\
    \frac{\dt}{4} \mathcal H_j (v^n-v^{(1)}, \varphi_h)+\frac{1}{4}(L_a,  \varphi_h), & j \in \cJ^n
  \end{array} \right. \\
\int_{I_j} e^{n+1} \varphi_h dx& =& \int_{I_j} \Upsilon \varphi_h dx+ \frac{1}{3} \int_{I_j} e^n \varphi_h dx+ \frac{2}{3} \int_{I_j} e^{(2)} \varphi_h dx , \quad \forall \varphi_h \in V_h \\
&& +\frac{2\dt}{3} \mathcal H_j (e^{(2)}, \varphi_h)+\left\{
  \begin{array}{l l}
    0,& j \notin \cJ^n\\
    \frac{2\dt}{3} \mathcal H_j (v^n-v^{(2)}, \varphi_h), & j \in \cJ^n
  \end{array} \right. 
\eeno
where $\Upsilon=v^{n+1}-v^{(3)}.$ Using the decomposition of errors \eqref{eqn:decompose}, we get
\begin{subequations}
\label{eqn:erroreqn}
\begin{align}
\int_{I_j} \xi^{(1)} \varphi_h dx =& \int_{I_j} \xi^n \varphi_h dx + \dt \mathcal J_j (\varphi_h), \quad \forall \varphi_h \in V_h\\
\int_{I_j} \xi^{(2)} \varphi_h dx =& \frac{3}{4} \int_{I_j} \xi^n \varphi_h dx+ \frac{1}{4} \int_{I_j} \xi^{(1)} \varphi_h dx  +\frac{\dt}{4} \mathcal K_j (  \varphi_h), \quad \forall \varphi_h \in V_h\\
\int_{I_j} \xi^{n+1} \varphi_h dx =&  \frac{1}{3} \int_{I_j} \xi^n \varphi_h dx+ \frac{2}{3} \int_{I_j} \xi^{(2)} \varphi_h dx  +\frac{2\dt}{3} \mathcal L_j ( \varphi_h) , \quad \forall \varphi_h \in V_h \end{align}
\end{subequations}
where

\begin{subequations}
\begin{align}
\mathcal J_j (\varphi_h) =& \int_{I_j} \frac{1}{\dt}  (\eta^{(1)}-\eta^n)\varphi_h dx + \mathcal H_j (e^n, \varphi_h), \quad \forall \varphi_h \in V_h \label{eqnj}\\
 \mathcal K_j (\varphi_h) =&   \int_{I_j} \frac{1}{\dt}  (4 \eta^{(2)}-3 \eta^n -\eta^{(1)}) \varphi_hdx+\mathcal H_j (e^{(1)}, \varphi_h), \quad \forall \varphi_h \in V_h \label{eqnk}\\
 &+\left\{
  \begin{array}{l l}
    0,& j \notin \cJ^n\\
     \mathcal H_j (v^n-v^{(1)}, \varphi_h)+\frac{1}{\dt}(L_a,  \varphi_h), & j \in \cJ^n
  \end{array} \right. \notag \\
 \mathcal L_j (\varphi_h) =&   \int_{I_j} \frac{1}{2\dt}  (3\eta^{n+1}- \eta^n-2 \eta^{(2)} +3\Upsilon) \varphi_hdx+\mathcal H_j (e^{(2)}, \varphi_h), \quad \forall \varphi_h \in V_h \label{eqnl}\\
 &+\left\{
  \begin{array}{l l}
    0,& j \notin \cJ^n\\
     \mathcal H_j (v^n-v^{(2)}, \varphi_h), & j \in \cJ^n
  \end{array} \right. \notag
   \end{align}
\end{subequations}
We further denote $\mathcal J(\varphi_h)=\sum_j   \mathcal J_j (\varphi_h)$, $\mathcal K(\varphi_h)=\sum_j   \mathcal K_j (\varphi_h)$, $\mathcal L(\varphi_h)=\sum_j   \mathcal L_j (\varphi_h)$.
By letting $\varphi_h=\xi^n, 4\xi^{(1)}, 6 \xi^{(2)}$ in \eqref{eqnj}, \eqref{eqnk}, \eqref{eqnl}, respectively, we get the following energy equation for $\xi^n$, \cite{Zhang_Shu_RK3_2010} 
\be
&&3\|\xi^{n+1}\|^2-3\|\xi^n\|^2=\dt [ \mathcal J (\xi^n)+\mathcal K(\xi^{(1)}) + \mathcal L(\xi^{(2)}  ] \\
  &&+\|2\xi^{(2)}-\xi^{(1)}-\xi^n\|^2+3(\xi^{n+1}-\xi^n, \xi^{n+1}-2\xi^{(2)}+\xi^n)\notag
\ee
Now we define $\Pi_1:=\dt [ \mathcal J (\xi^n)+\mathcal K(\xi^{(1)}) + \mathcal L(\xi^{(2)}  ] $, $\Pi_2:=\|2\xi^{(2)}-\xi^{(1)}-\xi^n\|^2+3(\xi^{n+1}-\xi^n, \xi^{n+1}-2\xi^{(2)}+\xi^n)$. We will estimate those two terms separately.

\medskip
\textbf{Estimate of $\Pi_1$}

Firstly, we notice that
\beno 
\dt \mathcal J (\xi^n)&=&(\eta^{(1)}-\eta^n, \xi^n) +\dt \mathcal H(e^n, \xi^n) \\
&=&(\eta^{(1)}-\eta^n, \xi^n)  +\dt \mathcal H(\xi^n, \xi^n) \\
&=&(\eta^{(1)}-\eta^n, \xi^n) -\frac{\dt}{2} \sum_j c [\xi^n]_{j+1/2}^2 \\
&\leq& \|\eta^{(1)}-\eta^n\| \cdot \|\xi^n\| -\frac{\dt}{2} \sum_j c [\xi^n]_{j+1/2}^2 \\
&\leq&  \frac{1}{4 \dt \epsilon} \|\eta^{(1)}-\eta^n\|^2 + \epsilon \dt \|\xi^n\|^2 -\frac{\dt}{2} \sum_j c [\xi^n]_{j+1/2}^2
\eeno
where in the second  line we have used the property of the Legendre-Gauss-Radau 
projection to get $\mathcal H(\eta^n, \varphi_h)=0$. In the formulas above, $\epsilon$ is a positive constant of order 1. Since 
$$
\eta^{(1)}-\eta^n=\pro (v^{(1)}-v^n)-(v^{(1)}-v^n)=-\dt \,c \, \left (\pro (v^n)_x -(v^n)_x \right),
$$
similar to Lemma \ref{lem:bound1}, we get
\beno
\|\pro (v^n)_x -(v^n)_x \| &\leq& \|\pro (v^n)_x -(v^n)_x \|_{\cB^n} +\|\pro (v^n)_x -(v^n)_x\|_{I\backslash \cB^n} \\
 &\leq& C h^{1/2}+ C h^{\min(\ell,k+1)} \leq C h^{1/2},
\eeno
 Therefore $\|\eta^{(1)}-\eta^n\| \leq C \dt h^{1/2}$ and
$$
\dt \mathcal J (\xi^n) \leq  C \dt h  + \epsilon \dt \|\xi^n\|^2 -\frac{\dt}{2} \sum_j c [\xi^n]_{j+1/2}^2.
$$

Similarly,
\beno 
&&\dt \mathcal K (\xi^{(1)})\\
&&= (4 \eta^{(2)}- \eta^{(1)}-3 \eta^n+L_a, \xi^{(1)}) +\dt \mathcal H(e^{(1)}, \xi^{(1)}) +\dt \sum_{j \in \cJ^n} \mathcal H_j(v^n-v^{(1)}, \xi^{(1)}) \\
&&=(4 \eta^{(2)}- \eta^{(1)}-3 \eta^n+L_a, \xi^{(1)})  -\frac{\dt}{2} \sum_j c [\xi^{(1)}]_{j+1/2}^2+\dt \sum_{j \in \cJ^n} \mathcal H_j(v^n-v^{(1)}, \xi^{(1)}) \\
&&\leq \frac{1}{4\dt \epsilon} \|4 \eta^{(2)}- \eta^{(1)}-3 \eta^n\|^2 +\frac{1}{4\dt \epsilon} \| L_a\|^2 + \frac{\epsilon}{2} \dt \|\xi^{(1)}\|^2 -\frac{\dt}{2} \sum_j c [\xi^{(1)}]_{j+1/2}^2\\
&&+\dt \sum_{j \in \cJ^n} \mathcal H_j(v^n-v^{(1)}, \xi^{(1)})
\eeno
Since $||L_a||_\infty \leq C \dt^2$, and $L_a \neq 0$ only on $\cB^n$, therefore $||L_a||  \leq C \dt^2 h^{1/2}.$
Next, we will estimate the term $\|4 \eta^{(2)}- \eta^{(1)}-3 \eta^n\|$ and $\dt \sum_{j \in \cJ} \mathcal H_j(v^n-v^{(1)}, \xi^{(1)})$.
We can derive that
\[
4 \eta^{(2)}- \eta^{(1)}-3 \eta^n=\left\{
  \begin{array}{l l}
    - \dt c (\pro v^{(1)}_x- v^{(1)}_x), & x \notin \cB^n\\
    - \dt c (\pro v^n_x- v^n_x)+(\pro L_a -L_a), & x \in \cB^n 
  \end{array} \right.
\]
%Because $||\pro L_a -L_a||_{\cB^n} \leq C h ||L_a||_{\cB^n} \leq C  \dt^2 h^{3/2}$, and 
% similar to the previous argument, we get $\|4 \eta^{(2)}- \eta^{(1)}-3 \eta^n\| \leq C \dt h^{1/2}.$ 
Because $L_a$ is a linear polynomial and $k \geq 1$, $\pro L_a -L_a=0$, and 
 similar to the previous argument, we get $\|4 \eta^{(2)}- \eta^{(1)}-3 \eta^n\| \leq C \dt h^{1/2}.$ 

As for $\dt \sum_{j \in \cJ} \mathcal H_j(v^n-v^{(1)}, \xi^{(1)})$, we have 
$$v^n-v^{(1)}=c \dt (v^n)_x,$$
therefore for any $j$
$$|(v^n-v^{(1)})_{j\pm1/2}| \leq C \dt \|v\|_{W^{1,\infty}}$$
and 
$$\|v^n-v^{(1)}\|_{I_j} \leq C \dt h^{1/2} \|v\|_{W^{1,\infty}}.$$
Hence
\beno
&&\mathcal H_j(v^n-v^{(1)}, \xi^{(1)})\\
&&=\int_{I_j} c (v^n- v^{(1)}) \xi^{(1)}_x dx - c (v^n-v^{(1)})^- (\xi^{(1)})_{j+1/2}^- + c (v^n-v^{(1)})^- (\xi^{(1)})_{j-1/2}^+\\
&&\leq C \dt h^{1/2} \|\xi^{(1)}_x\|_{I_j}+C \dt |(\xi^{(1)})_{j+1/2}^-|+C \dt |(\xi^{(1)})_{j-1/2}^+|\\
&& \leq C \dt h^{-1/2} \|\xi^{(1)}\|_{I_j}
\eeno
by inverse inequalities, and
\beno
\dt \sum_{j \in \cJ^n} \mathcal H_j(v^n-v^{(1)}, \xi^{(1)}) &\leq& C \dt^2 h^{-1/2} \|\xi^{(1)}\|_{\cB^n}\\
&\leq& C \dt^2 +\frac{\epsilon}{2} \dt \|\xi^{(1)}\|^2_{B^n}\\
&\leq& C \dt^2 +\frac{\epsilon}{2} \dt \|\xi^{(1)}\|^2,
\eeno
where in the second line we have used the  CFL condition $\dt \leq 
C_{cfl} h$. Putting everything together, and using the CFL condition 
again, we have
\beno 
&&\dt \mathcal K (\xi^{(1)})\leq C \dt h  +  \epsilon \dt \|\xi^{(1)}\|^2 -\frac{\dt}{2} \sum_j c [\xi^{(1)}]_{j+1/2}^2
\eeno

Finally,
\beno 
&&\dt \mathcal L (\xi^{(2)})\\
&&= \frac{1}{2} (3 \eta^{n+1}- 2\eta^{(2)}- \eta^n+ 3\Upsilon, \xi^{(2)}) +\dt \mathcal H(e^{(2)}, \xi^{(2)}) +\dt \sum_{j \in \cJ^n} \mathcal H_j(v^n-v^{(2)}, \xi^{(2)})\\
&&=  \frac{1}{2} (3 \eta^{n+1}- 2\eta^{(2)}- \eta^n+ 3\Upsilon, \xi^{(2)}) -\frac{\dt}{2} \sum_j c [\xi^{(2)}]_{j+1/2}^2 +\dt \sum_{j \in \cJ^n} \mathcal H_j(v^n-v^{(2)}, \xi^{(2)})\\
&&\leq  \frac{1}{8 \dt \epsilon} \|3 \eta^{n+1}- 2\eta^{(2)}- \eta^n\|^2+ \frac{9}{8\dt \epsilon} \|\Upsilon\|^2 + \epsilon \dt \|\xi^{(2)}\|^2 -\frac{\dt}{2} \sum_j c [\xi^{(2)}]_{j+1/2}^2\\
&&\quad +\dt \sum_{j \in \cJ^n} \mathcal H_j(v^n-v^{(2)}, \xi^{(2)}).
\eeno
Using the same argument as the previous terms, we have $\|3 \eta^{n+1}- 2\eta^{(2)}- \eta^n\| \leq C \dt h^{1/2}$ and $\dt \sum_{j \in \cJ^n} \mathcal H_j(v^n-v^{(2)}, \xi^{(2)}) \leq C \dt^2 +\epsilon \dt \|\xi^{(2)}\|^2$. 
As for $\Upsilon$, we have 
\beno
\|\Upsilon\|^2 &=& \int_{\cB^n} \Upsilon^2  dx + \int_{I\backslash \cB^n} \Upsilon^2 dx \\
& = &  \int_{\cB^n} (v^{n+1}-v^{(3)})^2  dx+ \int_{I\backslash \cB^n} \Upsilon^2 dx \\
& \leq &  \int_{\cB^n} (v^{n+1}-v^n+c \dt (v^n)_x-\frac{1}{6}L_a)^2  dx+(C \dt^4)^2 \\
&=& \int_{\cB^n} (v^{n+1}-v^n + c \dt (v^n)_x)^2  dx+C\dt^4h+(C \dt^4)^2\\
&\leq& C \dt^2 h+C\dt^4h+(C \dt^4)^2 \leq C \dt^2 h.
\eeno

Finally we obtain
\beno 
&&\dt \mathcal L (\xi^{(2)})\leq C \dt h  + \epsilon \dt \|\xi^{(2)}\|^2 -\frac{\dt}{2} \sum_j c [\xi^{(2)}]_{j+1/2}^2.
\eeno

Putting everything together, we have
\beno
\Pi_1 &\leq& C \dt h +\epsilon \dt \|\xi^n\|^2+\epsilon \dt \|\xi^{(1)}\|^2+\epsilon \dt \|\xi^{(2)}\|^2\\
&&-\frac{\dt}{2} \sum_j c [\xi^{n}]_{j+1/2}^2-\frac{\dt}{2} \sum_j c [\xi^{(1)}]_{j+1/2}^2-\frac{\dt}{2} \sum_j c [\xi^{(2)}]_{j+1/2}^2
\eeno

\medskip
\textbf{Estimate of $\Pi_2$}

To estimate $\Pi_2$, we first introduce
\beno
&&\mathbb{G}_1:=\xi^{(1)}-\xi^n\\
&&\mathbb{G}_2:=2 \xi^{(2)}-\xi^{(1)}-\xi^n\\
&&\mathbb{G}_3:=\xi^{n+1}-2 \xi^{(2)}+\xi^n.
\eeno

 From the error equation \eqref{eqn:erroreqn}, we can deduce
\begin{subequations}
\begin{align}
\int_{I_j} \mathbb{G}_1 \varphi_h dx =&   \dt \mathcal J_j (\varphi_h), \quad \forall \varphi_h \in V_h\\
\int_{I_j} \mathbb{G}_2 \varphi_h dx =&  \frac{\dt}{2} (\mathcal K_j (  \varphi_h) -\mathcal J_j (\varphi_h)), \quad \forall \varphi_h \in V_h\\
\int_{I_j} \mathbb{G}_3 \varphi_h dx =&  \frac{\dt}{3} (2\mathcal L_j ( \varphi_h) -\mathcal K_j ( \varphi_h)-\mathcal J_j ( \varphi_h) ), \quad \forall \varphi_h \in V_h \end{align}
\end{subequations}

Now, 
$$\Pi_2=(\mathbb{G}_2, \mathbb{G}_2)+3(\mathbb{G}_1, \mathbb{G}_3)+3(\mathbb{G}_2, \mathbb{G}_3)+3(\mathbb{G}_3, \mathbb{G}_3).$$
First, let us estimate $(\mathbb{G}_2, \mathbb{G}_2)+3(\mathbb{G}_1, \mathbb{G}_3).$
\beno
&&(\mathbb{G}_2, \mathbb{G}_2)+3(\mathbb{G}_1, \mathbb{G}_3)\\
&&=-\|\mathbb{G}_2\|^2 + 2 (\mathbb{G}_2, \mathbb{G}_2)+3(\mathbb{G}_1, \mathbb{G}_3)\\
&&=-\|\mathbb{G}_2\|^2 + \dt \left [ \mathcal K (\mathbb{G}_2)-\mathcal  J (\mathbb{G}_2)+2 \mathcal  L (\mathbb{G}_1) - \mathcal  K (\mathbb{G}_1) - \mathcal J (\mathbb{G}_1)  \right ]
\eeno
We have
\beno
&&\dt ( \mathcal K (\mathbb{G}_2)-\mathcal  J (\mathbb{G}_2))\\
&=&  (4 \eta^{(2)}-3\eta^n-\eta^{(1)}-(\eta^{(1)}-\eta^n)+L_a, \mathbb{G}_2) + \dt  \mathcal H (e^{(1)}-e^n, \mathbb{G}_2)\\
&&+\dt \sum_{j \in \cJ^n} \mathcal H_j(v^n-v^{(1)}, \mathbb{G}_2)\\
&=& (4 \eta^{(2)}-3\eta^n-\eta^{(1)}-(\eta^{(1)}-\eta^n)+L_a, \mathbb{G}_2 )+ \dt  \mathcal H (\mathbb{G}_1, \mathbb{G}_2)\\
&&+\dt \sum_{j \in \cJ^n} \mathcal H_j(v^n-v^{(1)}, \mathbb{G}_2) \\
&\leq & \|4 \eta^{(2)}-3\eta^n-\eta^{(1)}-(\eta^{(1)}-\eta^n)+L_a\|^2  +  \frac{1}{4}\|\mathbb{G}_2\|^2\\
&&+ \dt  \mathcal H (\mathbb{G}_1, \mathbb{G}_2) +\dt \sum_{j \in \cJ^n} \mathcal H_j(v^n-v^{(1)}, \mathbb{G}_2)\\
&\leq & C \dt^2 h  +  \frac{1}{4}\|\mathbb{G}_2\|^2+ \dt  \mathcal H (\mathbb{G}_1, \mathbb{G}_2) +\dt \sum_{j \in \cJ^n} \mathcal H_j(v^n-v^{(1)}, \mathbb{G}_2)\\
&\leq& C \dt^2 h  +  \frac{1}{4}\|\mathbb{G}_2\|^2+ \dt  \mathcal H (\mathbb{G}_1, \mathbb{G}_2)+ C \dt^2 h^{-1/2} \|\mathbb{G}_2\|\\
&\leq& C \dt^2 h  +  \frac{1}{4}\|\mathbb{G}_2\|^2+ \dt  \mathcal H (\mathbb{G}_1, \mathbb{G}_2) +C \dt^4 h^{-1} +  \frac{1}{4}\|\mathbb{G}_2\|^2\\
&\leq& C \dt^2 h  +  \frac{1}{2}\|\mathbb{G}_2\|^2+ \dt  \mathcal H (\mathbb{G}_1, \mathbb{G}_2).
\eeno

For the other term, similarly we get
\beno
&&\dt (2 \mathcal  L (\mathbb{G}_1) - \mathcal  K (\mathbb{G}_1) - \mathcal J (\mathbb{G}_1)  )\\
&=&  (3\eta^{n+1}-2 \eta^{(2)}-\eta^n+3\Upsilon-(4\eta^{(2)}-3\eta^n-\eta^{(1)})-(\eta^{(1)}-\eta^n)-L_a, \mathbb{G}_1)\\
&& + \dt  \mathcal H (2 e^{(2)}-e^{(1)}-e^n, \mathbb{G}_1)+\dt \sum_{j \in \cJ^n} \mathcal H_j(2(v^n-v^{(2)})-(v^n-v^{(1)}), \mathbb{G}_1)\\
&\leq& C \dt h + \dt  \mathcal H (\mathbb{G}_2, \mathbb{G}_1) + \frac{\epsilon}{2} \dt \|\mathbb{G}_1\|^2\\
&\leq& C \dt h + \dt  \mathcal H (\mathbb{G}_2, \mathbb{G}_1) + \epsilon \dt \|\xi^n\|^2 +\epsilon \dt \|\xi^{(1)}\|^2.
\eeno
Therefore,
\beno
&&(\mathbb{G}_2, \mathbb{G}_2)+3(\mathbb{G}_1, \mathbb{G}_3)\\
&&\leq-\frac{1}{2}\|\mathbb{G}_2\|^2 +   C \dt h+ \dt  \mathcal H (\mathbb{G}_2, \mathbb{G}_1)+\dt  \mathcal H (\mathbb{G}_1, \mathbb{G}_2)+ \epsilon \dt \|\xi^n\|^2 +\epsilon \dt \|\xi^{(1)}\|^2\\
&&\leq-\frac{1}{2}\|\mathbb{G}_2\|^2 +   C \dt h- \dt  \sum_j c [\mathbb{G}_1] [ \mathbb{G}_2]+ \epsilon \dt \|\xi^n\|^2 +\epsilon \dt \|\xi^{(1)}\|^2\\
&&\leq-\frac{1}{2}\|\mathbb{G}_2\|^2 +   C \dt h+ \frac{\dt}{4}  \sum_j c |[\mathbb{G}_1]|^2 +\dt  \sum_j c |[\mathbb{G}_2]|^2+ \epsilon \dt \|\xi^n\|^2 +\epsilon \dt \|\xi^{(1)}\|^2.
\eeno
Also 
\beno
&&3(\mathbb{G}_3, \mathbb{G}_2)\\
&=& \dt (   2\mathcal L ( \mathbb{G}_2) -\mathcal K ( \mathbb{G}_2)-\mathcal J (\mathbb{G}_2) )\\
&=& (3\eta^{n+1}-2 \eta^{(2)}-\eta^n+3\Upsilon-(4\eta^{(2)}-3\eta^n-\eta^{(1)})-(\eta^{(1)}-\eta^n)-L_a, \mathbb{G}_2 )\\
&& + \dt  \mathcal H (2 e^{(2)}-e^{(1)}-e^n, \mathbb{G}_2)+\dt \sum_{j \in \cJ^n} \mathcal H_j(2(v^n-v^{(2)})-(v^n-v^{(1)}), \mathbb{G}_2)\\
&\leq& C \dt h+ \dt  \mathcal H (\mathbb{G}_2, \mathbb{G}_2)+ \epsilon \dt \|\mathbb{G}_2\|^2\\
&\leq& C \dt h-\frac{\dt}{2} \sum_j c [\mathbb{G}_2]_{j+1/2}^2+ \epsilon \dt \|\mathbb{G}_2\|^2.
\eeno
and
\beno
&&3\|\mathbb{G}_3\|^2=3(\mathbb{G}_3, \mathbb{G}_3)\\
&=& \dt (   2\mathcal L ( \mathbb{G}_3) -\mathcal K ( \mathbb{G}_3)-\mathcal J (\mathbb{G}_3) )\\
&=& (3\eta^{n+1}-2 \eta^{(2)}-\eta^n+3\Upsilon-(4\eta^{(2)}-3\eta^n-\eta^{(1)})-(\eta^{(1)}-\eta^n)-L_a, \mathbb{G}_3)\\
&& + \dt  \mathcal H (2 e^{(2)}-e^{(1)}-e^n, \mathbb{G}_3)+\dt \sum_{j \in \cJ^n} \mathcal H_j(2(v^n-v^{(2)})-(v^n-v^{(1)}), \mathbb{G}_3)\\
&\leq& C \dt h^{1/2} \|\mathbb{G}_3\|+C \frac{\dt}{h} \|\mathbb{G}_2\|\cdot \|\mathbb{G}_3\|+ C  \dt^2 h^{-1/2} \|\mathbb{G}_3\|
\eeno
Therefore $$\|\mathbb{G}_3\|\leq C \dt h^{1/2}+C \|\mathbb{G}_2\|$$ due to the CFL condition and
$$3\|\mathbb{G}_3\|^2\leq C \dt^2 h+\frac{1}{4}\|\mathbb{G}_2\|^2.$$

Putting everything together, we have
\beno
\Pi_2
 &\leq& (-\frac{1}{4}+\epsilon \dt) \|\mathbb{G}_2\|^2+ C \dt h + \frac{\dt}{4} \sum_j c [\mathbb{G}_1]_{j+1/2}^2\\
 &&+\frac{\dt}{2} \sum_j c [\mathbb{G}_2]_{j+1/2}^2+ \epsilon \dt \|\xi^n\|^2 +\epsilon \dt \|\xi^{(1)}\|^2\\
 & \leq& (-\frac{1}{4}+\epsilon \dt) \|\mathbb{G}_2\|^2+ C \dt h + \frac{\dt}{2} \sum_j c( [\xi^n]_{j+1/2}^2+[\xi^{(1)}]_{j+1/2}^2)\\
 &&+C \frac{\dt}{h} \|\mathbb{G}_2\|^2+ \epsilon \dt \|\xi^n\|^2 +\epsilon \dt \|\xi^{(1)}\|^2\\
 & \leq& (-\frac{1}{4}+\epsilon \dt+C C_{cfl}) \|\mathbb{G}_2\|^2+C \dt h+ \frac{\dt}{2} \sum_j c( [\xi^n]_{j+1/2}^2+[\xi^{(1)}]_{j+1/2}^2)\\
 &&+ \epsilon \dt \|\xi^n\|^2 +\epsilon \dt \|\xi^{(1)}\|^2
 \eeno
When $\epsilon$ and $C_{cfl}$ are small enough, $-\frac{1}{4}+\epsilon \dt+C C_{cfl} \leq 0$, and
 $$\Pi_2 \leq C \dt h+ \frac{\dt}{2} \sum_j c( [\xi^n]_{j+1/2}^2+[\xi^{(1)}]_{j+1/2}^2)
+ \epsilon \dt \|\xi^n\|^2 +\epsilon \dt \|\xi^{(1)}\|^2.$$

Finally
 $$\Pi_1+\Pi_2 \leq  C\dt h+2\epsilon \dt \|\xi^n\|^2+2\epsilon \dt \|\xi^{(1)}\|^2+\epsilon \dt \|\xi^{(2)}\|^2.$$
 
At this point, we need to provide an estimate of $\|\xi^{(1)}\|$, $\|\xi^{(2)}\|$ to finish the proof. Plug in the error equation \eqref{eqn:erroreqn},
 \beno
 &&\|\xi^{(1)}\|^2=(\xi^{(1)},\xi^{(1)})=(\xi^n, \xi^{(1)})+ \dt \mathcal J(\xi^{(1)})\\
 && \leq \|\xi^{n}\|\cdot\|\xi^{(1)}\| +(\eta^{(1)}-\eta^n, \xi^{(1)}) + \dt \mathcal H (\xi^n-\eta^n, \xi^{(1)})\\
 && \leq \|\xi^{n}\|\cdot\|\xi^{(1)}\| + C \dt h^{1/2} \|\xi^{(1)}\|+ C\frac{\dt}{h}(\|\xi^n\|+\|\eta^n\|) \|\xi^{(1)}\|\\
  && \leq C\|\xi^{n}\|\cdot\|\xi^{(1)}\| + C \dt h^{1/2} \|\xi^{(1)}\|
 \eeno
 Therefore, $$ \|\xi^{(1)}\| \leq C\|\xi^n\|+ C \dt h^{1/2}.$$
 Similarly, $$ \|\xi^{(2)}\| \leq C\|\xi^n\|+C \|\xi^{(1)}\| + C \dt h^{1/2} \leq C\|\xi^n\|+ C \dt h^{1/2}.$$
 
 Overall,  $$\Pi_1+\Pi_2 \leq  C\dt h+C \dt \|\xi^n\|^2,$$ and
 \beno
3\|\xi^{n+1}\|^2-3\|\xi^n\|^2 \leq C\dt h+C \dt \|\xi^n\|^2
\eeno
i.e.
\be
\label{rkroot}
\|\xi^{n+1}\|^2 \leq (1+ C \dt) \|\xi^n\|^2 + C\dt h
\ee
and by induction with the initial condition satisfying $||\xi^0|| \leq C h^\qkell$,
$$
\|\xi^n\|  \leq C h^{1/2} 
$$
and we are done using the projection property $\|\eta^n\|   \leq C h^\qkell $, since $\qkell=\frac{3}{2}$ in this case.

The final bound is 
\beno
  \|u_h^{n} - u^{n}\|  \leq  C\, (h+\frac{h^3}{\dt^2})^{1/2}.
\eeno
In the case $h/\dt$ is bounded from below, we obtain a bound of order $h^{1/2}$.
\end{proof}

\subsection{Convergence of RKDG scheme in the obstacle case}
Now we turn to scheme \eqref{eq:DGscheme} for  the obstacle equation \eqref{eq:02}. 
In particular,
the scheme writes: initialize with 
$ u_h^0:=\Pi_h u_0$, and 
$u_h^{n+1} = \Pi_h (\max(G_{\dt}^{RK} (u_h^n) , \tilde g))$ for $n \geq 0$. The main idea follows closely the proof of Theorem \ref{th:2}, but utilizes the estimates in Theorem \ref{th:rkdg11}.

\begin{thm}\label{th:rkdg2}
Assume that the exact solution is not shattering in the sense of Definition~\ref{def:noshat}.
Under the same assumption as in Theorem \ref{th:rkdg11} (in particular 
the assumption on the CFL condition)
with both $\tilde g= g_\dt$ and $\tilde g=\max\big( g(\cdot), g(\cdot-c\dt) \big)$, 
scheme \eqref{eq:DGscheme} with RKDG solver $G_{\dt}^{RK}$ satisfies
\beno
  \|u_h^{n} - u^{n}\|  \leq  C\, (h+\frac{h^3}{\dt^2})^{1/2}
\eeno
for some constant $C\geq 0$ independent of $h, \dt, u_h$.

In particular, if $\dt/h$ is bounded from below ($\frac{\dt}{h} \geq 
\bar C_0$ for some constant $\bar C_0 >0$), we have
\beno
  \|v_h^{n} - v^{n}\|   \leq  C_1 h^{1/2}.
\eeno
\end{thm}

\begin{proof}
Similar to the proof of Theorem \ref{th:2}, we can obtain
\beno
 && \|e^{n+1}\|=\|u_h^{n+1} - u^{n+1}\|
    \leq  \| G_{\dt}^{RK} (u_h^n) - u^n(\cdot - c\dt)\| + \| \tilde g -  g_\dt \| +  C h^\qkell.
\eeno
We decompose the error $$  u_h^{n} - u^n =\eta^n-\xi^n,$$
where $\xi^n=\pro u^{n}- u_h^{n}, \quad \eta^n=\pro u^{n}-u^{n},$ and
$$G_{\dt}^{RK} (u_h^n) - u^n(\cdot - c\dt) =\xi'-\eta'$$
where $\xi'=\pro u^n(\cdot - c\dt)-G_{\dt}^{RK} (u_h^n), \quad \eta'=\pro u^n(\cdot - c\dt)-u^n(\cdot - c\dt).$

Therefore
\be
 \|e^{n+1}\|
    &\leq&  \| \xi'\| + \| \eta'\| + \| \tilde g -  g_\dt \| +  C h^\qkell \nonumber \\
    &\leq&  \| \xi'\|   + \| \tilde g -  g_\dt \| +  C h^\qkell \label{eq:A}
\ee
by using the projection property again. Hence
\be
 \|e^{n+1}\|^2
   & \leq & (1+\epsilon) \| \xi'\|^2 + (1+\frac{1}{\epsilon}) \left( \| \tilde g -  g_\dt \| +  C h^\qkell \right)^2
 \label{eq:B}
\ee

Using \eqref{rkroot} in the proof of Theorem \ref{th:rkdg11}, we obtain
\beno
  \| \xi'\|^2 \leq (1+C \dt) \|\xi^n\|^2+ C \dt h
\eeno
Hence
\beno
  \|e^{n+1}\|^2 & \leq &
    (1+\epsilon) \left((1+C \dt) \|\xi^n\|^2+ C \dt h \right)+ (1+\frac{1}{\epsilon}) 
    \left( \| \tilde g -  g_\dt \| +  C h^\qkell \right)^2
\eeno
Now we take  $\epsilon =C \dt$, 
\beno
  \|e^{n+1}\|^2
   & \leq&   (1+C \dt) \|\xi^n\|^2+ C \dt h +C \dt^{-1} \left( \| \tilde g -  g_\dt \| +  C h^\qkell \right)^2\\
      & \leq&   (1+C \dt) \|\xi^n\|^2+ C \dt h +C  \dt^{-1} h^3 \\
   &\leq& (1+\epsilon) (1+C \dt) \|e^n\|^2+(1+\frac{1}{\epsilon}) (1+C \dt) \|\eta^n\|^2+ C \dt h+C  \dt^{-1} h^3 \\
   &\leq&  (1+C \dt) \|e^n\|^2  + C \dt h+C  \dt^{-1} h^3
\eeno
where in the last line we have taken $\epsilon=C \dt$ again. By induction on $n$,
 we are done.
\end{proof}

\subsection{Convergence of the RKDG scheme defined with Gauss-Legendre quadrature points}
%---------------------------------------------------
Now we turn to scheme \eqref{eq:DGscheme2} for  the nonlinear equation \eqref{eq:02}. 
In particular,
the scheme writes: initialize with 
$ u_h^0:=\Pi_h u_0$, and 
 $u_h^{n+1}$ is  defined as the unique polynomial in $V_h$ such that :
%\be\label{eq:sldg_obs}
   %u^{n+1}(x^i_\ma):= \max(\Pi_h (u^n( \cdot - c\dt))(x^i_\ma),\ g_\dt(x^i_\ma)), \quad \forall \ma,i.
\be\label{eq:rkdg_obs}
   u^{n+1}_h(\xja):= \max(G_{\dt}^{RK}  (u_h^{n})(x^j_\ma),\ \tilde g(x^j_\ma)), \quad \forall j, \,\ma.
\ee

\begin{thm}\label{th:rkdg3}
Assume that the exact solution is not shattering in the sense of Definition~\ref{def:noshat}.
Under the same assumption as in Theorem \ref{th:rkdg11} (in particular assuming the CFL condition) 
with both $\tilde g= g_\dt$ and $\tilde g(x):=\max\big( g(x), g(x-c\dt) \big)$,   scheme \eqref{eq:DGscheme2} with RKDG solver  $G_{\dt}^{RK}$ satisfies
\beno
  \|u_h^{n} - u^{n}\|  \leq  C\, (h + \frac{h^3}{\dt^2})^{1/2}
\eeno
for some constant $C\geq 0$ independent of $h, \dt, u_h$.

In particular, if $\dt/h$ is bounded from below ($\frac{\dt}{h} \geq 
\bar C_0$ for some constant $\bar C_0 >0$), we have
\beno
  \|v_h^{n} - v^{n}\|   \leq  C_1 h^{1/2}.
\eeno
\end{thm}

\begin{proof}
Following the same lines as the proof for Theorem \ref{th:3}, we obtain
\beno
    \|e^{n+1}\|=\| u_h^{n+1} - u^{n+1} \| 
  & \leq & 
 \|G_{\dt}^{RK}  (u_h^{n})-\Pi_h u^n(\cdot - c \dt)\|_{\ell^2} +\|\tilde g -g_\dt\|_{\ell^2} + C  h^\qkell \\
  & =& 
 \|G_{\dt}^{RK} (u_h^{n})-\Pi_h u^n(\cdot - c \dt)\| +\|\tilde g -g_\dt\|_{\ell^2} + C  h^\qkell   \\
 & \leq & 
 \|G_{\dt}^{RK} (u_h^{n})- u^n(\cdot - c \dt)\| +\|\tilde g -g_\dt\|_{\ell^2} + C  h^\qkell.
\eeno
By the same argument as in Theorem \ref{th:rkdg2}
\beno
  \|e^{n+1}\|^2
   & \leq&   (1+C \dt) \|\xi^n\|^2+ C \dt h +C \dt^{-1} \left( \| \tilde g -  g_\dt \|_{\ell^2} +  C h^\qkell \right)^2\\
      & \leq&   (1+C \dt) \|\xi^n\|^2+ C \dt h +C  \dt^{-1} h^3 \\
   &\leq&  (1+C \dt) \|e^n\|^2  + C \dt h+C  \dt^{-1} h^3,
\eeno
and we are done.
\end{proof}

%From the theorems above, we can see that if $k \geq 1, l \geq2$,  the 
%RKDG methods for the obstacle problem converge if $h=o(\dt^{2/3})$, 
%subject also to the CFL condition $\dt \leq C_0 h$.
%That includes the case of $h=O(\dt)$, for which the convergence rate 
%is $h^{1/2}$.

%-----------------------
%- Numerical section here
%-----------------------
%\input{num.tex}

\section{Numerical results}\label{sec:num}
%-----------------------------------------

In this section we consider  one- and a two-dimensional examples to validate our results.
In the one dimensional setting the SLDG scheme and the RKDG scheme are tested
and in the two dimensional setting only the RKDG scheme is tested.

%For RKDG schemes, we always use the RK3 version.
%In one dimensional and two dimensional setting, 
%we perform accuracy test for the solution away from singular points. For two dimensional examples, we are mainly
%concerned with evolution of the zero level set.

\smallskip

\paragraph{\bf Example~1 (1--d)}
This is a one-dimensional test (same as \cite[Example 1]{bok-che-shu-14}):
\be  & &  \min(u_t +  u_x , u - g(x)) =0,\quad  t>0,\ x\in [-1,1],\\
     & &  u(0,x) = u_0(x), \quad x\in [-1,1],
\ee
with periodic boundary conditions and $g(x):=\sin(\pi x)$, $u_0(x):=0.5+\sin(\pi x)$.
In that case, for times $0\leq t\leq 1$, the exact solution is given by :
\newcommand{\tzero}{\frac{1}{3}}
$$
 u^{(1)}(t,x)= \left\{ \barr{ll}
   \max( u_0(x-t),\ g(x)) & \mbox{ if } t<\tzero,\\
   \max( u_0(x-t),\ g(x),\ 1_{x\in[0.5,1]} ) & \mbox{ if } t\in [\tzero,\tzero+\ud], \\
   \max( u_0(x-t),\ g(x),\ 1_{x\in[-1,t-\frac{1}{3}-\frac{1}{2}]\cup [0.5,1]} )
     & \mbox{ if } t\in [\tzero+\ud,1].\\
  \earr \right.
$$
%where $t_0=\frac{1}{3}$. %$t_0:=0.5-\frac{1}{\pi}\arcsin(0.5)$.

We first consider the RKDG scheme~\eqref{eq:DGscheme2}. 
The choice $\dt=0.2\, h$ is made considering the
stability of the RKDG scheme with $P^2$ elements.

In Figure~\ref{fig:1d-1} the numerical solution is shown, which
agrees well with the exact solution everywhere.
Table~\ref{tab:1d-1} contains the numerical errors at time $t=0.5$.
The errors are computed globally (a uniform grid mesh of $10^4$ points is used in each mesh cell to estimate the errors).
In this example there are three singular points ($s_1\simeq-0.1349$, $s_2:=0.5$ and $s_3=2/3$) where the solution is continuous
but with different left and right derivatives. We use a least square procedure to calculate the approximate order of the scheme, see Figure \ref{fig:1d-error1}. From the calculation, the orders for $L^1, L^2, L^\infty$ errors are $1.75, 1.35$ and $0.97$, respectively.

Next, in Table~\ref{tab:1d-2} and in Table~\ref{tab:1d-3}, numerical errors for the SLDG scheme are given at time  $t=0.5$.
Table~\ref{tab:1d-2} shows the results for the choice of time step $\dt=h/2$.
Then, in Table~\ref{tab:1d-3} the choice of $\dt$ of the order of $h^{3/5}$ is made,
as suggested in the theoretical study (number of time steps is $10 (N_x/10)^{3/5}$) and since there is no CFL restriction on this scheme. The $L^1, L^2, L^\infty$ order of the methods based on the least square plots Figure \ref{fig:1d-error2} are $1.66, 1.21, 0.72$ (when $\dt=h/2$) and $1.48, 1.37, 0.95$ (when $\dt=C\,h^{3/5}$).

We can clearly observe from this example that the numerical orders
of convergence 
in the $L^\infty$ norm are at least as good as and often better
than those obtained by our analysis.

%The order of the method is still not clear but one can observe an improvement of the results with respect to the previous Table~\ref{tab:1d-2}
%for the $L^2$ and the $L^\infty$ norms.

%we have computed all error norms in the region
%$[-1,1] \backslash \bigg(\disp\bigcup_{i=1,3} [s_i-\delta,s_i+\delta]\bigg)$
%where $s_1:=-0.1349733$, $s_2:=0.5$ and $s_3=2/3$
%are the three singular points of the solution at time $t=0.5$, and with $\delta=0.1$. We observe the optimal third order convergence rate for $P^2$ elements.
\begin{figure}[!hbtp]
\begin{center}
\includegraphics[width=0.32\textwidth]{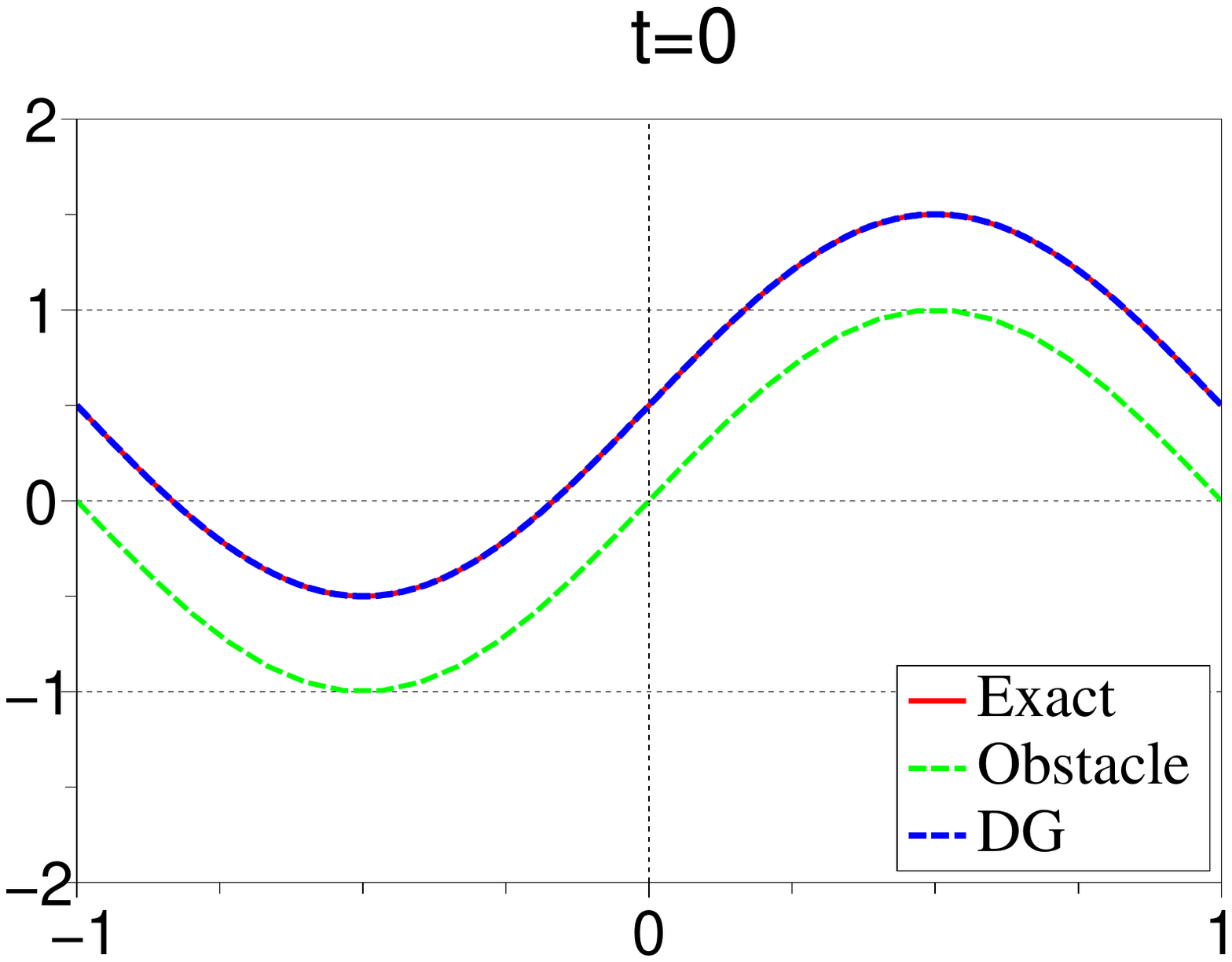}
\includegraphics[width=0.32\textwidth]{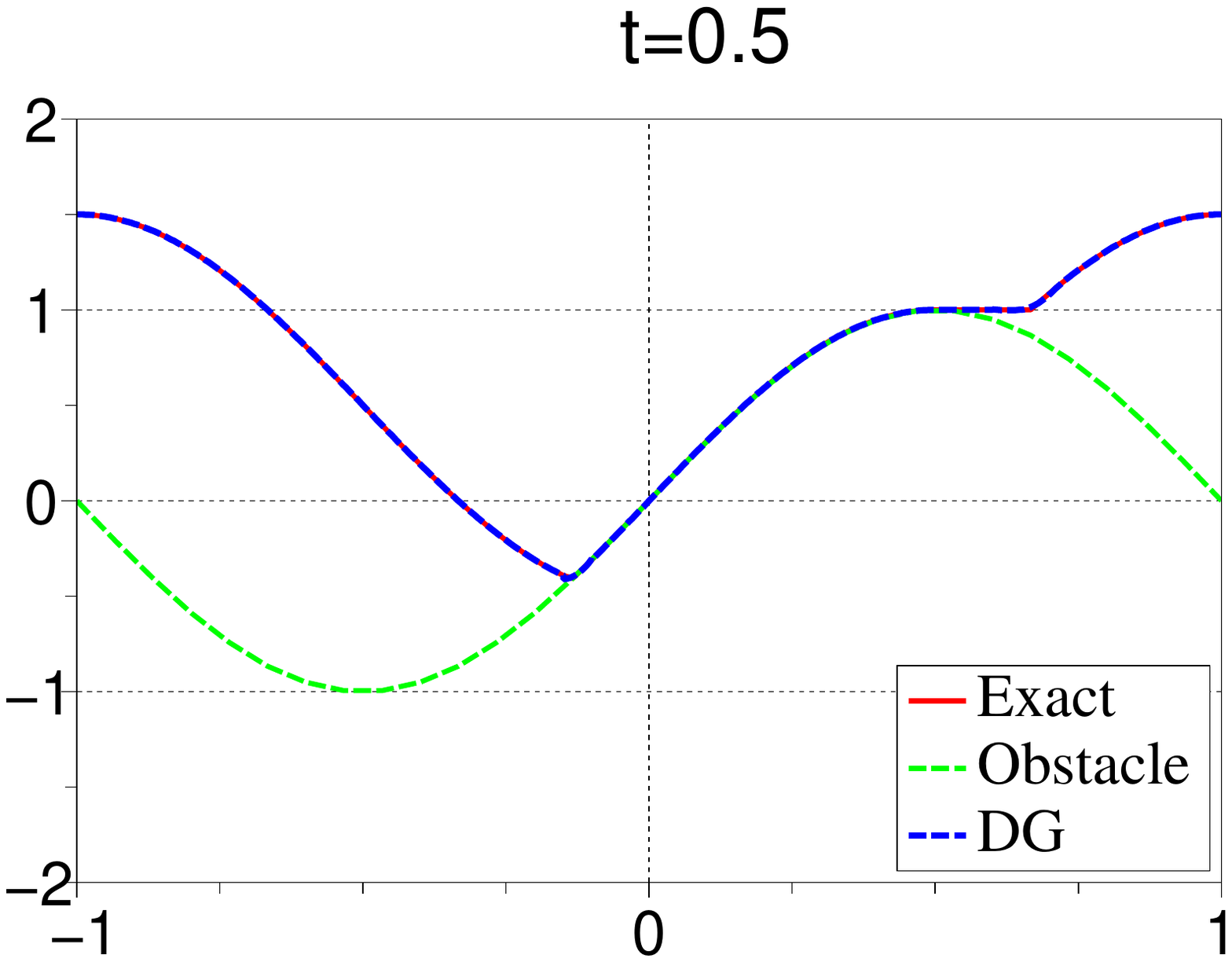}
\includegraphics[width=0.32\textwidth]{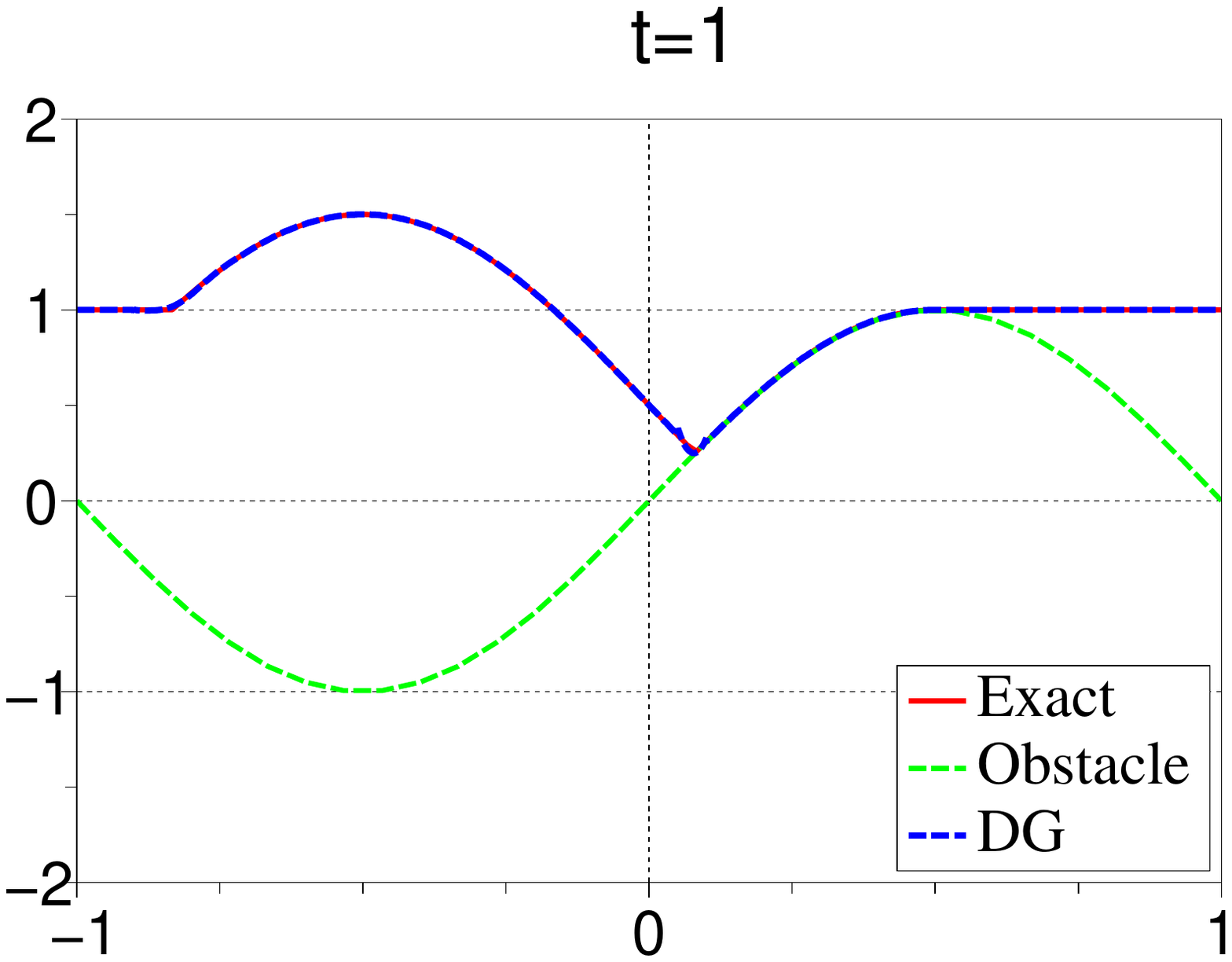}
\vspace*{-2cm}
\caption{\label{fig:1d-1} Example~1, RKDG scheme, times~$t=0$, $t=0.5$ and~$t=1$, using $P^2$ elements with $N_x=20$ mesh cells
(obstacle : green dotted line).}
\end{center}
\end{figure}

\begin{table}[!hbtp]
\caption{\label{tab:1d-1} Example~1, RKDG scheme with $P^2$ elements.}
%\begin{tabular}{rl|cc|cc|cc}
%$N_x$ & $\dx$   & $L^1$-error & order & $L^2$-error &order & $L^\infty$-error& order\\[.1cm] \hline\hline
%% 10  & 2.00e-1 & 1.18e-03  &   -   &  1.74e-03 &   -    &  4.95e-03  &   -   \\
%% 20  & 1.00e-1 & 1.77e-04  &  2.74 &  3.99e-04 &  2.12  &  3.23e-03  &  0.62  \\
% 40  & 5.00e-2 & 3.34e-05  &  2.41 &  1.01e-04 &  1.98  &  7.02e-04  &  2.20  \\
% 80  & 2.50e-2 & 1.77e-06  &  4.24 &  3.64e-06 &  4.79  &  2.82e-05  &  4.64  \\
% 160 & 1.25e-2 & 1.78e-07  &  3.31 &  2.91e-07 &  3.64  &  2.40e-06  &  3.55  \\
% 320 & 6.25e-3 & 2.13e-08  &  3.06 &  3.43e-08 &  3.08  &  1.28e-07  &  4.23  \\
% 640 & 3.13e-3 & 2.66e-09  &  3.00 &  4.28e-09 &  3.00  &  1.60e-08  &  3.00  \\
%1280 & 1.56e-3 & 3.32e-10  &  3.00 &  5.35e-10 &  3.00  &  2.00e-09  &  3.00
\begin{tabular}{r|cc|cc|cc}
$N_x$ & $L^1$-error & order & $L^2$-error &order & $L^\infty$-error& order\\[.1cm] \hline\hline
%   nx  -    L1    ,   o1   -    L2    ,   o2   -  Linf    , oinf     
%   20  & 3.42E-03 &   -    & 8.28E-03 &   -    & 6.27E-02 &   -    \\ 
%   40  & 7.31E-04 &  2.23  & 2.15E-03 &  1.94  & 1.96E-02 &  1.67  \\ 
%   80  & 2.80E-04 &  1.38  & 1.15E-03 &  0.91  & 1.53E-02 &  0.36  \\ 
%- HERE CFL=0.02
%   80  & 2.80E-04 &  -     & 1.15E-03 &  -     & 1.53E-02 &  -     \\ 
%  160  & 5.93E-05 &  2.24  & 3.39E-04 &  1.76  & 7.55E-03 &  1.01  \\ 
%  320  & 1.80E-05 &  1.72  & 1.19E-04 &  1.51  & 2.36E-03 &  1.68  \\ 
%  640  & 5.92E-06 &  1.60  & 5.32E-05 &  1.17  & 1.31E-03 &  0.85  \\ 
% 1280  & 1.96E-06 &  1.59  & 2.44E-05 &  1.12  & 9.84E-04 &  0.41  \\ 
% 2560  & 5.20E-07 &  1.92  & 8.67E-06 &  1.50  & 4.31E-04 &  1.19  \\ 
% 5120  & 1.61E-07 &  1.69  & 3.50E-06 &  1.31  & 2.30E-04 &  0.91  \\ 
%10240  & 4.81E-08 &  1.75  & 1.43E-06 &  1.29  & 1.29E-04 &  0.83  
%- HERE CFL=0.2
%   nx  -    L1    ,   o1   -    L2    ,   o2   -  Linf    , oinf     
%   10  & 1.08E-02 &   -    & 1.53E-02 &   -    & 5.20E-02 &   -    \\ 
%   20  & 3.10E-03 &  1.81  & 6.84E-03 &  1.16  & 4.92E-02 &  0.08  \\ 
%   40  & 7.67E-04 &  2.01  & 2.12E-03 &  1.69  & 1.50E-02 &  1.71  \\ 
   80  & 2.68E-04 &  1.52  & 1.01E-03 &  1.07  & 1.25E-02 &  0.27  \\ 
  160  & 6.47E-05 &  2.05  & 3.46E-04 &  1.55  & 6.80E-03 &  0.87  \\ 
  320  & 1.96E-05 &  1.72  & 1.30E-04 &  1.41  & 2.52E-03 &  1.43  \\ 
  640  & 6.40E-06 &  1.62  & 5.65E-05 &  1.21  & 1.43E-03 &  0.82  \\ 
 1280  & 2.10E-06 &  1.61  & 2.54E-05 &  1.15  & 8.26E-04 &  0.79  \\ 
 2560  & 6.14E-07 &  1.77  & 1.02E-05 &  1.32  & 4.76E-04 &  0.79  \\ 
 5120  & 1.98E-07 &  1.63  & 4.35E-06 &  1.22  & 2.75E-04 &  0.79  \\ 
10240  & 6.19E-08 &  1.68  & 1.88E-06 &  1.21  & 1.61E-04 &  0.78  \\ 
\end{tabular}
\end{table}

\begin{figure}[!hbtp]
\begin{center}
\includegraphics[width=0.8\textwidth]{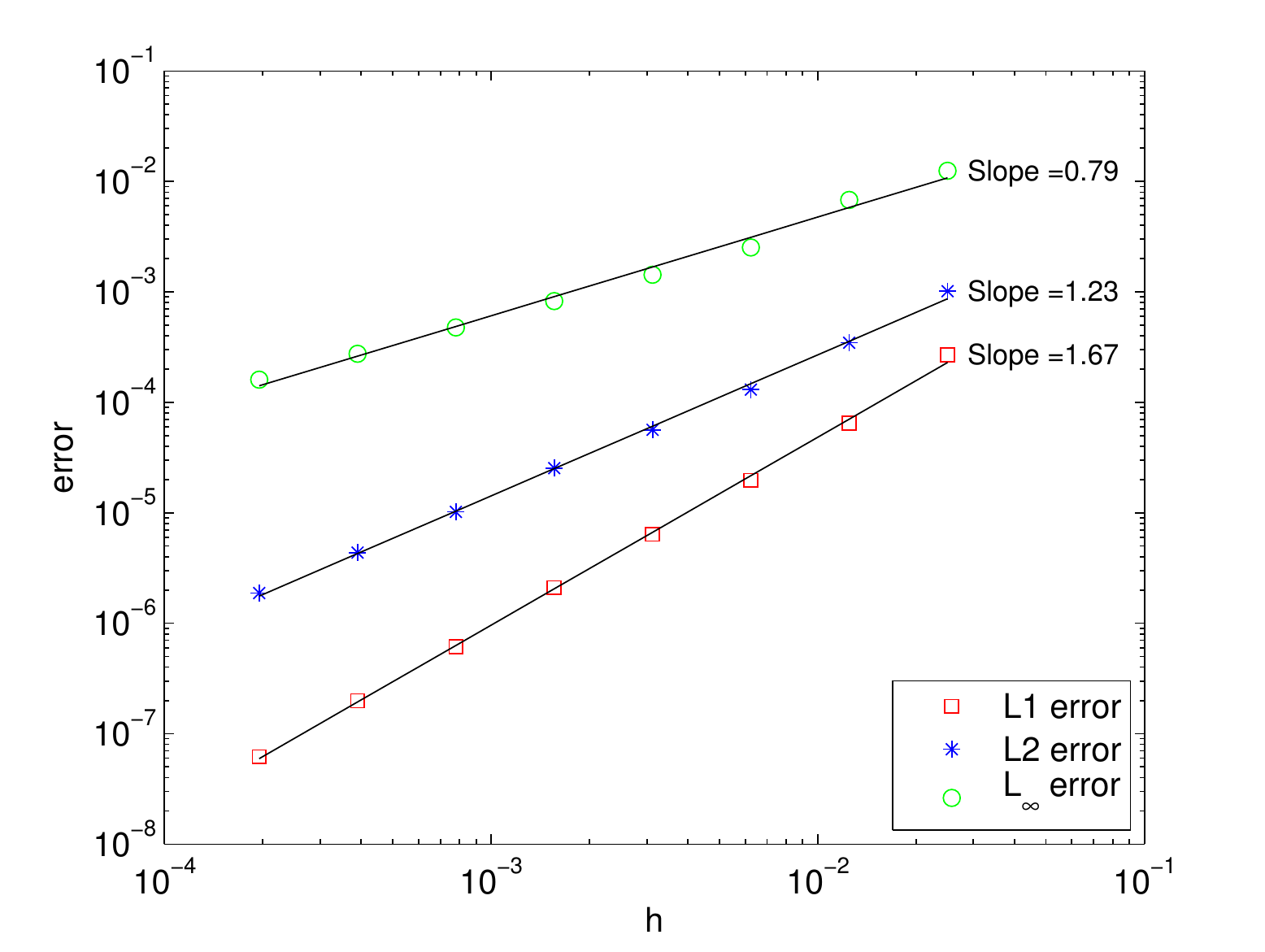}
\caption{\label{fig:1d-error1} Example~1, RKDG scheme with $P^2$ elements. }
\end{center}
\end{figure}

% We have tested $P0$, $P1$ and $P^2$ approach in 1--d for the obstacle problem.
% It utilizes a classical 1--d code with $P^k$ DG elements
% and in the polynomial basis $(x^j)_{j=0,\dots,k}$.
% The projection step needs to precompute a matrix that passes from
% the coefficients in the polynomial basis set, to the coefficients in the
% Gaussian basis set. It also needs the inverse of the matrix.
% This is computed in {\em "van1"} and {\em "van1i"} (Vandermonde matrix and its inverse), in file
% {\em "init\_obstacle.f"} (as well as obstacle values $g(x_j)$ at gauss points).
% The projection subroutine is in file {\em "projection.f"}

%\paragraph{\bf Example 1 (1--d, SLDG)}

\begin{table}[!hbtp]
\caption{\label{tab:1d-2} Example~1, SLDG scheme, with $P^2$ elements and $\dt=h/2$.}
\begin{tabular}{r|cc|cc|cc}
% SLDG CODE 
$N_x$ & $L^1$-error & order & $L^2$-error &order & $L^\infty$-error& order\\[.1cm] \hline\hline
%-
%-  TEST (M=N) 
%-
%   M         L1      order        L2     order       Linf    order
%    10 &  1.15E-03 &  0.00  & 1.42E-03 &  0.00  & 2.29E-03 &  0.00  \\ 
%    20 &  1.90E-03 & -0.73  & 4.09E-03 & -1.53  & 1.19E-02 & -2.38  \\ 
%    40 &  2.30E-04 &  3.04  & 5.40E-04 &  2.92  & 1.80E-03 &  2.72  \\ 
%    80 &  1.73E-04 &  0.41  & 6.21E-04 & -0.20  & 3.11E-03 & -0.79  \\ 
    80 &  1.73E-04 &    -   & 6.21E-04 &   -    & 3.11E-03 &   -    \\ 
   160 &  2.38E-05 &  2.86  & 1.21E-04 &  2.35  & 8.58E-04 &  1.86  \\ 
   320 &  1.56E-05 &  0.61  & 9.77E-05 &  0.31  & 8.26E-04 &  0.05  \\ 
   640 &  4.73E-06 &  1.72  & 3.54E-05 &  1.47  & 4.30E-04 &  0.94  \\ 
  1280 &  1.73E-06 &  1.45  & 1.94E-05 &  0.87  & 3.25E-04 &  0.41  \\ 
  2560 &  4.11E-07 &  2.07  & 6.79E-06 &  1.52  & 1.76E-04 &  0.88  \\ 
  5120 &  1.31E-07 &  1.65  & 3.03E-06 &  1.16  & 1.19E-04 &  0.57  \\ 
 10240 &  4.03E-08 &  1.70  & 1.22E-06 &  1.32  & 6.67E-05 &  0.83  \\ 
\end{tabular}
\end{table}

\begin{table}[!hbtp]
\caption{\label{tab:1d-3} Example~1, SLDG scheme, with $P^2$ elements and $\dt=C\,h^{3/5}$.}
\begin{tabular}{rr|cc|cc|cc}
$N_x$ & $N$ &  $L^1$-error & order & $L^2$-error &order & $L^\infty$-error& order\\[.1cm] \hline\hline
%   M       N           L1      order         L2      order         Linf    order
%   M       N           L1      order         L2      order         Linf    order
%- N=10* (M/10)^(3/5)
%   M       N           L1      order         L2      order         Linf    order
%    10 &    10  & 1.15E-03 &  0.00  & 1.42E-03 &  0.00  & 2.29E-03 &  0.00  \\ 
%    20 &    15  & 1.40E-03 & -0.29  & 2.82E-03 & -0.99  & 6.93E-03 & -1.60  \\ 
%    40 &    22  & 1.55E-04 &  3.18  & 4.12E-04 &  2.77  & 1.58E-03 &  2.13  \\ 
    80 &    34  & 6.04E-05 &  1.36  & 1.79E-04 &  1.20  & 8.26E-04 &  0.94  \\ 
   160 &    52  & 1.99E-05 &  1.60  & 9.34E-05 &  0.94  & 7.58E-04 &  0.12  \\ 
   320 &    79  & 7.95E-06 &  1.33  & 2.66E-05 &  1.81  & 2.70E-04 &  1.49  \\ 
   640 &   121  & 2.50E-06 &  1.67  & 1.43E-05 &  0.90  & 1.86E-04 &  0.54  \\ 
  1280 &   183  & 9.80E-07 &  1.35  & 7.45E-06 &  0.94  & 1.46E-04 &  0.34  \\ 
  2560 &   278  & 3.66E-07 &  1.42  & 2.93E-06 &  1.35  & 8.03E-05 &  0.86  \\ 
  5120 &   422  & 1.22E-07 &  1.58  & 7.41E-07 &  1.98  & 2.74E-05 &  1.55  \\ 
 10240 &   639  & 4.48E-08 &  1.45  & 1.83E-07 &  2.02  & 5.82E-06 &  2.23  
\end{tabular}
\end{table}

\begin{figure}[!hbtp]
\begin{center}
\includegraphics[width=0.49\textwidth]{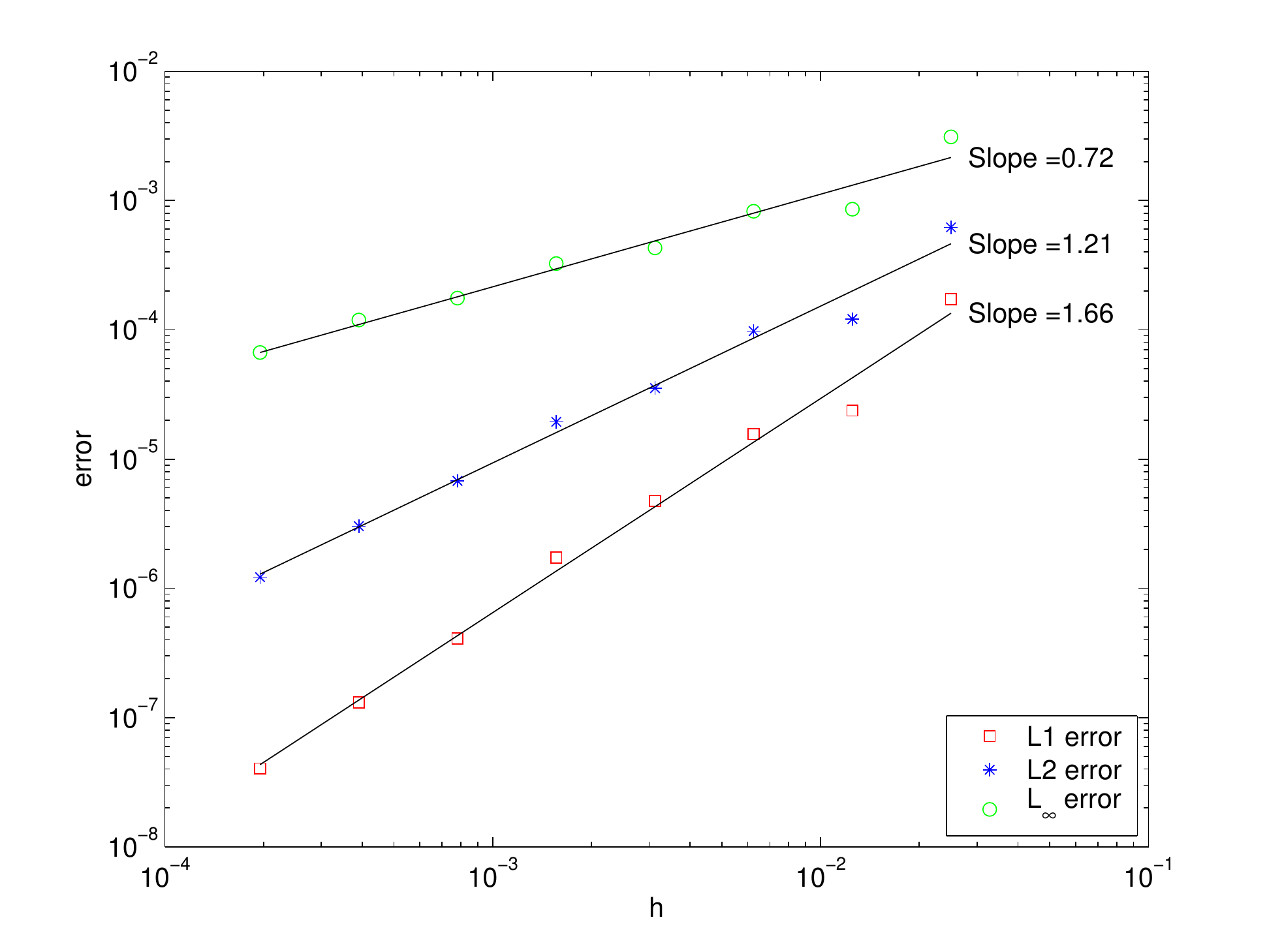}
\includegraphics[width=0.49\textwidth]{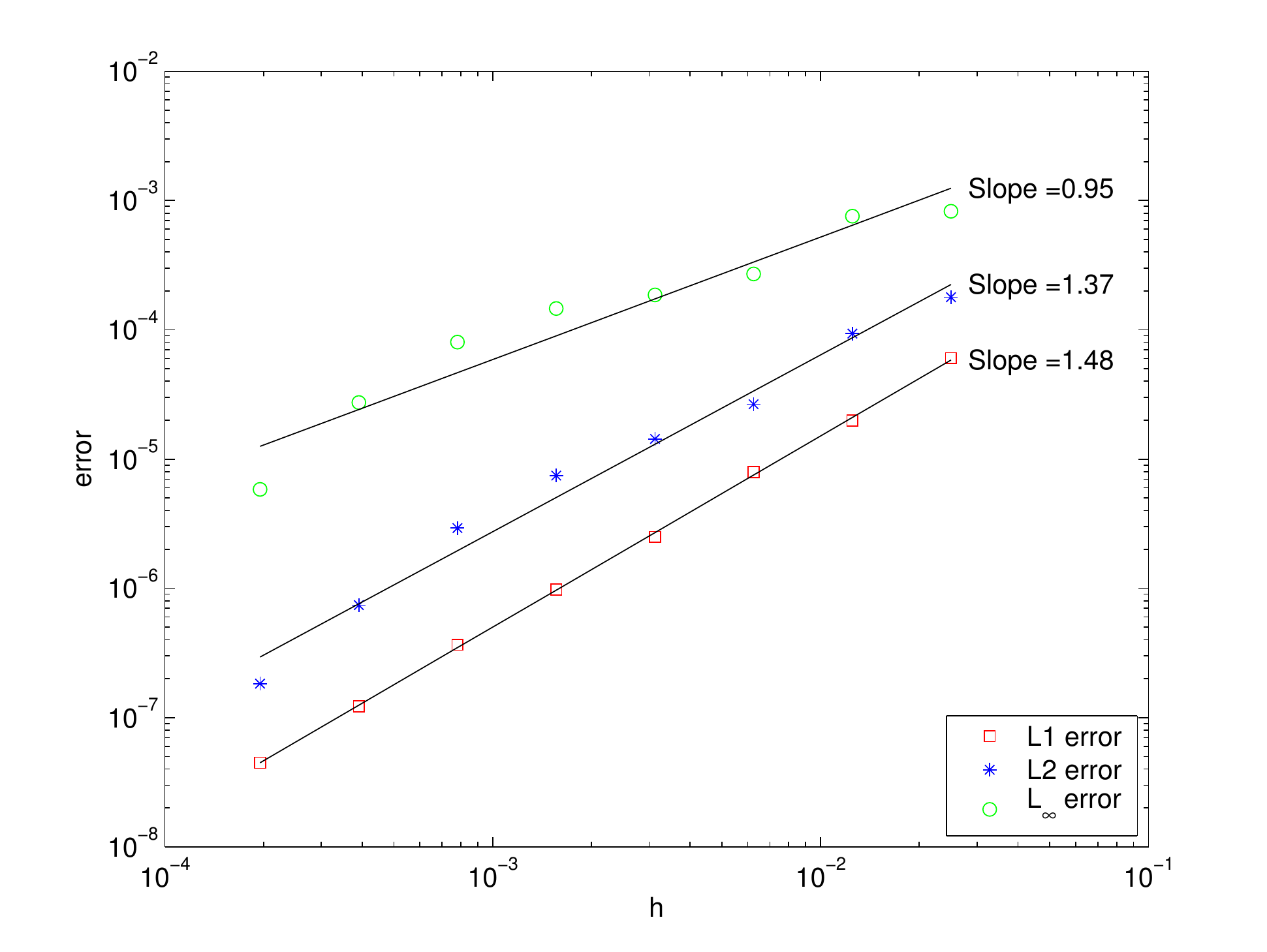}
\caption{\label{fig:1d-error2} Example~1, SLDG scheme, with $P^2$ elements. Left: $\dt=h/2$, Right: $\dt=C\,h^{3/5}$. }
\end{center}
\end{figure}

\paragraph{\bf Example 2 (2--d, RKDG)}
This is a two-dimensional test, same as \cite[Example 3]{bok-che-shu-14}).
The equation solved is
\be  & &  \min(u_t + \frac{1}{2} u_{x} +\frac{1}{2} u_{y}, u - g(x,y)) =0,\quad  t>0,\ (x,y)\in \mO,\\
     & &  u(0,x,y) = u_0(x,y), \quad (x,y)\in \mO,
\ee
where $g(x,y):= \sin(\pi (x + y))$, $u_0(x,y)= 0.5 + g(x,y)$, and
$\mO=[-1,1]^2$ with periodic boundary conditions.
The exact solution is known and is obtained as in Example~1:
$$ u(t,x,y)= u^{(1)}(t,x+y). $$

We now consider the two-dimensional version of the RKDG scheme~\eqref{eq:DGscheme2}, using $Q^2$ elements (tensor product $P^2\otimes P^2$).

Accuracy results are shown in Table~\ref{tab:2d_lin1} and Figure \ref{fig:2d-error} for time $t=0.5$.
The errors are computed globally (a uniform grid mesh of $50^2$ points is used in each mesh cell to estimate the errors).

The results in this example demonstrate that our scheme is also
convergent in two-dimensions.  We will comment upon the
two-dimensional case in the next section.

%that is, only in the region
%%$\{(x,y)\in \mO, \  \forall i=1,2,3,\ |(x+y)[2]-s_i| \geq \delta\}$, with $\delta=0.1$ and $t=0.5$.
%$\{(x,y)\in \mO, \  1\leq i\leq 3,\ d( x+y-s_i, 2\Z)\geq \delta)\}$, with $\delta=0.1$. We observe optimal convergence rate in this example.

%\centerline{\fbox{$pp=50$}}

\begin{table}[!hbtp]
\caption{\label{tab:2d_lin1} Example 2, $Q^2$ elements.}
\begin{tabular}{cc|cc|cc|cc}
$N_x=N_y$ & $h_x = h_y$ & $L^1$-error & order & $L^2$-error &order & $L^\infty$-error& order\\[.1cm] \hline\hline
%10 & 2.00e-1 & 7.70e-03  &   -   &  1.03e-02 &   -    &  1.04e-01  &   -    \\
%20 & 1.00e-1 & 9.27e-04  &  3.05 &  1.28e-03 &  3.01  &  8.71e-03  &  3.58  \\
%40 & 5.00e-2 & 9.48e-05  &  3.29 &  1.67e-04 &  2.94  &  1.04e-03  &  3.06  \\
%80 & 2.50e-2 & 7.15e-06  &  3.73 &  1.11e-05 &  3.91  &  1.02e-04  &  3.34  \\
%errors and orders: (FOR LATEX TABLE): 
%   nx  -    L1    ,   o1   -    L2    ,   o2   -  Linf    , oinf     
10  & 2.00e-1 & 2.25E-02 &   -    & 2.27E-02 &   -    & 1.68E-01 &   -    \\ 
20  & 1.00e-1 & 6.70E-03 &  1.75  & 9.85E-03 &  1.20  & 1.24E-01 &  0.44  \\ 
40  & 5.00e-2 & 1.70E-03 &  1.98  & 3.29E-03 &  1.58  & 4.35E-02 &  1.51  \\ 
80  & 2.50e-2 & 5.11E-04 &  1.73  & 1.31E-03 &  1.33  & 2.67E-02 &  0.70  \\ 
160 & 1.25e-2 & 1.46E-04 &  1.80  & 5.19E-04 &  1.33  & 9.28E-03 &  1.52  \\ 
\end{tabular}
\end{table}

\begin{figure}[!hbtp]
\begin{center}
\includegraphics[width=0.8\textwidth]{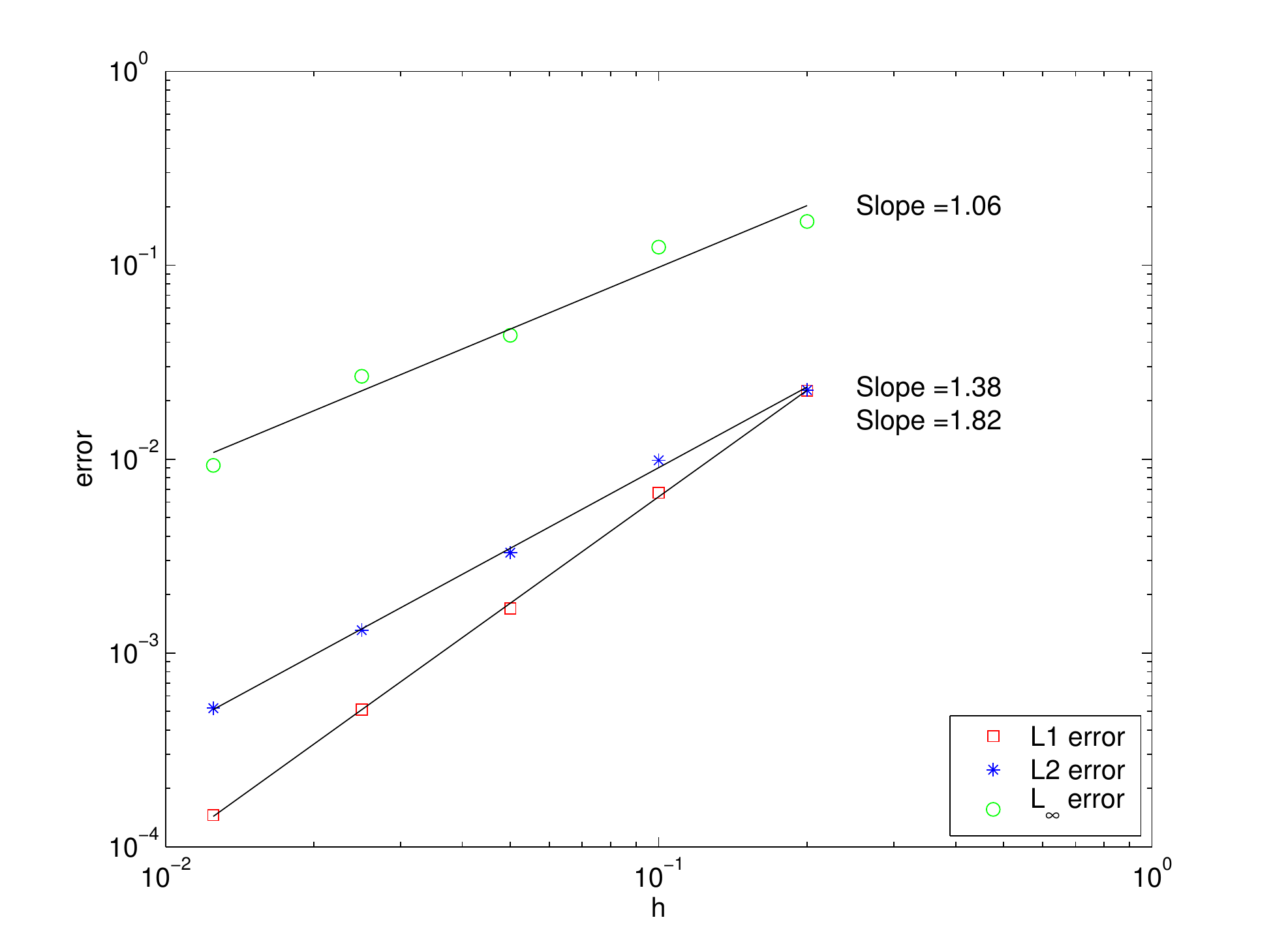}
\caption{\label{fig:2d-error} Example 2, $Q^2$ elements. }
\end{center}
\end{figure}

The numerical solution and the exact solution are also plotted
in Figure~\ref{fig:2d}, showing good agreements.

\begin{figure}[!hbtp]
\begin{center}
\includegraphics[width=0.45\textwidth]{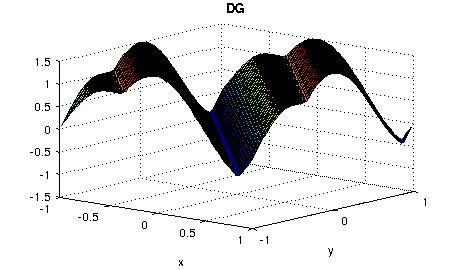}
\includegraphics[width=0.45\textwidth]{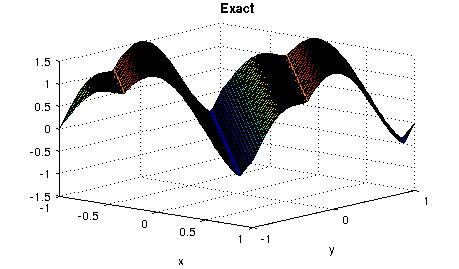}
\caption{\label{fig:2d} Example~2 at time~$t=0.5$, numerical (left) and exact (right) data, using $Q^2$ elements with $N_x=20$ mesh cells}
\end{center}
\end{figure}

\section{Concluding remarks}
\label{sec:conclusion}

In this paper, we prove convergence of the SLDG and RKDG methods for 
the obstacle problem under the ``no shattering" assumption of the 
exact solution. We utilize the DPP of the obstacle solutions. 
The proof of the SLDG methods relies on the property of the 
$L^2$ projection, while new techniques of devising piecewise 
intermediate stage functions are developed for the convergence of 
the RKDG methods. 

We remark that the proposed methods can be easily extended to 
treat multi-dimensional obstacle equations.
The simplest way is to define a multi-dimensional DG basis 
obtained as a tensor product of the one-dimensional DG basis.
The Gaussian points on each cell can then be defined accordingly. 
The operation of taking the maximum at the Gaussian points 
is therefore straightforward as in the one-dimensional case. 
The definition of the RKDG scheme in the multi-dimensional
case is well known and is an extension of the one-dimensional case.
On the other hand, 
the definition of the SLDG scheme in multi-dimensions
is not straightforward. However, high-order stable splittings 
methods for SLDG for the linear advection equation can be 
devised (see for example Bokanowski and Simarmata 
\cite{Bokanowski_Simarmata}). Finally the ``no-shattering" property
can be easily extended to multi-dimensions by demanding that the 
exact solution be piecewise regular except on a finite union of 
compact submanifolds,
and error estimates of the same order as in the one-dimensional
case will then hold.
We refer to Example 2 in the previous section for the numerical
performance of our RKDG methods in two-dimensions.

\bibliographystyle{abbrv}
\bibliography{bib_cheng,bib_comp}

\end{document}